\documentclass{amsart}

\usepackage{setspace}

\usepackage{appendix}
\usepackage{geometry}
\usepackage{amssymb}
\usepackage{amsmath}
\usepackage{amsthm}
\usepackage{mathtools}
\usepackage{mathrsfs}
\DeclareFontFamily{U}{rsfs}{\skewchar\font127 }
\DeclareFontShape{U}{rsfs}{m}{n}{%
   <-6> rsfs5
   <6-8> rsfs7
   <8-> rsfs10
}{}
\usepackage{bm}
\usepackage{hyperref}
\hypersetup{
	colorlinks=true,
	linkcolor=cyan!50!blue,
	filecolor=blue,      
	urlcolor=red,
	citecolor=green,
}
\usepackage{cleveref}
\usepackage[numbers,sort]{natbib}
\usepackage{framed}
\usepackage{graphicx}
\usepackage{xcolor}
\usepackage{listings}
\usepackage{multirow}
\usepackage{multicol}
\usepackage{indentfirst}
\usepackage{booktabs}
\usepackage{tikz}
\usepackage{float}
\usepackage{algorithm}
\usepackage{algorithmicx}
\usepackage{algpseudocode}
\usepackage{subfigure}
\usepackage{lineno}
\usepackage{ulem}
\usepackage{epstopdf}
\usetikzlibrary{patterns, arrows.meta, calc}

\newcommand*{\be}[1]{\begin{equation}\label{#1}}
\newcommand*{\ee}{\end{equation}}

\DeclareMathOperator{\card}{Card}
\DeclareMathOperator{\Span}{span}
\makeatletter
\@addtoreset{equation}{section}
\makeatother
\renewcommand{\theequation}{\arabic{section}.\arabic{equation}}

\definecolor{pink}{RGB}{255, 179, 195}

\DeclareMathOperator{\grad}{grad}

\DeclareMathOperator{\curl}{curl}

\DeclareMathOperator{\diverenge}{div}
\DeclareMathOperator{\im}{im}

\newcommand{\Pcal}{\mathcal{P}}
\renewcommand{\div}{\diverenge}
\renewcommand{\emptyset}{\varnothing}

\newcommand*{\lrangle}[1]{\left\langle#1\right\rangle}
\newtheorem{theorem}{Theorem}[section]
\newtheorem{lemma}[theorem]{Lemma}

\theoremstyle{definition}
\newtheorem{definition}[theorem]{Definition}
\newtheorem{example}[theorem]{Example}
\newtheorem{proposition}[theorem]{Proposition}

\newtheorem{assumption}{Assumption}

\theoremstyle{remark}
\newtheorem{remark}[theorem]{Remark}

\numberwithin{equation}{section}

\begin{document}

\title{A Construction of $C^r$ conforming finite element spaces in any dimension}

\author{Jun Hu}
\address{LMAM and School of Mathematical Sciences, Peking University, Beijing 100871,
P. R. China.}

 \email{hujun@math.pku.edu.cn}

\author{Ting Lin}
\address{School of Mathematical Sciences, Peking University, Beijing 100871,
P. R. China.}
\email{lintingsms@pku.edu.cn}

\author{Qingyu Wu}
\address{School of Mathematical Sciences, Peking University, Beijing 100871,
P. R. China.}
\email{wu\_qingyu@pku.edu.cn}



\begin{abstract}
    This paper proposes a construction of $C^r$ conforming finite element spaces with arbitrary $r$ in any dimension. It is shown that if $k \ge 2^{d}r+1$ the space $\mathcal P_k$ of polynomials of degree $\le k$ can be taken as the shape function space of $C^r$ finite element spaces in $d$ dimensions. This is the first work on constructing such $C^r$ conforming finite elements in any dimension in a unified way. 
\end{abstract}

\maketitle

\section{Introduction}
\label{sec:intro}
This paper is to provide an $H^{r+1}$ conforming finite element method of $2(r+1)$-th order elliptic problems on simplicial triangulations in $\mathbb{R}^d$. The \textit{conforming finite element method} is to seek for piecewise polynomial function spaces with global $C^{r}$ continuity. The commonly used $H^1$ conforming element is the celebrated $C^0$ Lagrange element on simplicial triangulations in $d$ dimensions. While the Hermite element is still $C^0$ conforming but cannot admit higher global continuity when $d>1$. The case $d = 1$ is an exception in the sense that the one-dimensional Hermite element is $H^2$ conforming.
The construction of $C^r$ conforming finite elements on simplicial triangulations in $d$ dimensions is a long-standing open problem \cite{xu2020finite}. One main difficulty is the choice of the shape function space. It is commonly conjectured that the shape function space can be chosen as the space $\mathcal{P}_k$ of polynomials of degree not greater than $k$ with $k \ge 2^dr+1$. However, no successful construction can be found in the literature. Another main difficulty is the design of degrees of freedom. Indeed, the traditional bubble function technique does not work anymore in general.

Many efforts have been made for this problem and partial results can be found in Bramble and Zl\'{a}mal \cite{bramble1970triangular}, where two dimensional $C^r$ elements on triangular grids were constructed for any $r \ge 0$ with $k = 4r+1$. For the case $r = 1$, it recovers the $\mathcal{P}_5-C^1$ Argyris element \cite{argyris1968tuba}. For the three-dimensional case, the first $\mathcal{P}_9-C^1$ element on tetrahedral grids was constructed by \v Zen\' i\v sek in \cite{vzenivsek1970interpolation}. Later, a family of $\mathcal{P}_{8r + 1}-C^r$ elements on tetrahedral grids was constructed for any $r \ge 0$ in Lai and Schumaker \cite{lai2007spline}, which recovers the $\mathcal{P}_9-C^1$ \v Zen\' i\v sek element for the case $r = 1$. Zhang \cite{zhang2009family} also extended the Zen\' i\v sek element into a $\mathcal P_{k}-C^1$ element family for $k\ge 9$. A family of $\mathcal{P}_k -C^2$ elements on tetrahedral grids and a family of $\mathcal{P}_k-C^1$ elements on 4-simplices with $k\ge 17$ were proposed in Zhang \cite{zhang2016family}, where the bubble function spaces were tailored.

On the contrary, the construction of $C^r$ conforming finite elements on $d$-cube grids is much easier. In fact, a family of $C^r$ conforming finite elements was designed in Hu and Zhang \cite{hu2015minimal} on macro-$d$-cube grids by using the space $\mathcal Q_{r+1}$ consisting of polynomials of degree $\le r+1$ in each variable, see also Hu, Huang and Zhang \cite{hu2011lowest} for $C^1$ conforming finite elements on macro-$d$-cube grids.

Given the difficulty of constructing $C^r$ conforming elements in $d$ dimensions, an alternative way is to weaken the continuity and to construct $H^{r+1}$ \textit{nonconforming elements}. In this direction, the first and very elegant construction in any dimension is from Wang and Xu \cite{wang2013minimal}, where nonconforming finite elements on simplicial triangulations were proposed and analyzed for $d \ge r+1$ by using the space $\mathcal{P}_{r+1}$ as the shape function space. For the case $r = 1$, it recovers  a very famous nonconforming element, namely, the Morley element, of fourth order problems \cite{morley1968triangular,ming2006morley}. That family was later extended to the case $r = d$ by enriching the full polynomial space $\mathcal{P}_{r+1}$ with higher order bubbles in Wu and Xu \cite{wu2019nonconforming}. Recently, a family of $H^{r+1}$ nonconforming finite elements was established by employing an interior penalty technique for the case $r+1>d$, using the space $\mathcal{P}_{r+1}$ as the shape function space, see Wu and Xu \cite{wu2017pmip}. While in Hu and Zhang \cite{hu2017canonical}, a family of two dimensional $H^{r+1}$ nonconforming elements was constructed on triangular grids, using the space $\mathcal{P}_{2r-1}$ as the shape function space when $r>2$.

If non-polynomials are considered as the shape functions, the \textit{virtual element method} \cite{beirao2014hitchhiker,beirao2013basic} can be used to design both conforming and nonconforming approximations of the space $H^{r+1}$ in any dimension. The interested readers can refer to Antonietti, Manzini and Verani \cite{antonietti2020conforming}, Chen and Huang \cite{chen2020nonconforming}, Huang \cite{huang2020nonconforming2} for relevant virtual element methods.

The neural network is a new method for discretization on partial differential equations. In Xu \cite{xu2020finite}, the \textit{finite neuron method} was proposed, utilizing a generalized ReLU neural network architecture to propose a conforming approximation of the space $H^{r+1}$ for any $r$ in $d$ dimensions.

A relevant topic of $C^r$ finite element methods is spline interpolations or supersplines \cite{chui1990multivariate, alfeld1992dimension}. In Chui and Lai \cite{chui1990multivariate}, the authors first constructed a family of vertex $C^r$ splines on simplicial triangulations in two dimensions by using piecewise polynomials of degrees not greater than $k$ with $k \geq 4r+1$, which is in fact a variant of the Bramble--Zl\'amal family~\cite{bramble1970triangular}. Then they formally extended their approach to constructing vertex $C^r$ splines on simplicial triangulations in any dimension by using piecewise polynomials of degrees not more than $k$ with $k\geq 2^dr+1$. In Alfeld, Schumaker, and Sirvent~\cite{alfeld1991structure}, it was proved that when the polynomial of degrees $k \ge 2^dr+1$ there exists a structure of the spline spaces that allow the construction of a minimally supported basis. However, no degrees of freedom were proposed therein, which is a key of finite element methods. The reference \cite{alfeld1992dimension} is a following-up paper of \cite{alfeld1991structure}, whose main result is, by using Berstein--B\'ezeir techniques, to show the existence of the local basis of spline spaces on simplicial triangulations for $d \le 3$ with $k \ge 2^dr+1$. The construction of the supersplines on the Alfeld or Powell--Sabin split in two and three dimensions can be found in \cite{lai2007spline}.

\subsection{Main result}

In this work, a family of $C^r$ conforming finite element spaces on simplicial triangulations is proposed, using the piecewise polynomials of degree not greater than $k$, with $k \ge 2^dr+1$. The construction generalizes all the conforming elements on simplicial triangulations introduced above, except the $\mathcal P_{8r+1}- C^r$ family introduced in \cite{lai2007spline}.
The main result is summarized in the following theorem.

\begin{theorem}
    \label{thm:informal}
    Given $u \in \mathcal{P}_{k}(K)$, the space of polynomials of degree not greater than $k \ge 2^d r +1 $ over $d$-dimensional simplex $K$, for any $(d-s)$-dimensional subsimplex $F$ of the simplicial triangulation $\mathcal{T}$ ($0\le s \le d$), let $$D^{\theta}:= \frac{\partial^n}{\partial \bm n_1^{\theta_1} \cdots\partial \bm n_s^{\theta_s}} \text{ with } n:= \sum_{i = 1}^{s} \theta_i $$ represent an $n$-th order normal derivative of $u$ on $F$ when $n > 0$, where $\bm n_1,\cdots, \bm n_s$ are orthonormal unit normal vector(s) of $F$. Define the following weighted moments

    \begin{equation}\label{eq:dof-informal} \frac{1}{\vert F \vert}\int_{F} (D^{\theta}u) \cdot v, \quad \forall v \in \mathcal B_{F,n,k}, \text{ if $F$ is not a vertex,} \end{equation}
   
    \begin{equation}\label{eq:dof-informal2}D^{\theta}u(F) \cdot v(F), \quad \forall v \in \mathcal B_{F,n,k}, \text{ if $F$ is a vertex,}
    \end{equation}
    where the bubble space
    
    \begin{equation}
    \label{eq:Bfnk-informal}
    \begin{split}
    \mathcal B_{F,n,k} &:= \Span \big\{ \prod_{i = s}^d \lambda_{F,i}^{\sigma_i}: (\sigma_s,\cdots,\sigma_d) \text{ satisfies \eqref{eq:thm-informal-cond1} and \eqref{eq:thm-informal-cond2}} \big\}  \\ & =  \Span \{ \lambda_{F,s}^{\sigma_s}\cdots \lambda_{F,d}^{\sigma_{d}} : (\sigma_s,\cdots,\sigma_d) \text{ satisfies \eqref{eq:thm-informal-cond1} and \eqref{eq:thm-informal-cond2}} \} 
    \end{split}%
\end{equation}
with $\lambda_{F,i}, i = s,\cdots, d$, the barycenter coordinates of $F$, and the multi-indices $(\sigma_s,\cdots, \sigma_{d})$ satisfy that \begin{equation}\label{eq:thm-informal-cond1}\sum_{i = s}^{d} \sigma_i = k - n \end{equation} and  \begin{equation}\label{eq:thm-informal-cond2}\sigma_{i_1}+\cdots+\sigma_{i_l} > 2^{l+s-1} r - n, ~~\forall \{i_1, \cdots, i_l\} \subsetneq \{s, \cdots, d\},\end{equation} with $ l = 1, \cdots, d-s, $  and $0 \le n \le 2^{s-1}r$. When $s = 0$, let $n = 0$. Henceforth, the production $\prod_{i = s}^d \lambda_{F,i}^{\sigma_i}$ will be simply shortened as $\lambda_{F,s}^{\sigma_s}\cdots\lambda_{F,d}^{\sigma_d}$ for convenience.

    Then this set of degrees of freedom is unisolvent for the shape function space $\mathcal{P}_{k}(K)$, and the resulting finite element space is of $C^r$ continuity.
\end{theorem}

\begin{remark}
Note that the degrees of freedom defined in \eqref{eq:dof-informal} and \eqref{eq:dof-informal2} are conventional and convenient for presentation. The more precise statement of the degrees of freedom of \eqref{eq:dof-informal} should be $\frac{1}{\vert F \vert}\int_{F} (D^{\theta}u) \cdot v_i$, where $v_i, i = 1,2,\cdots,\dim \mathcal B_{F,n,k}$, form a basis of $\mathcal B_{F,n,k}$. 
\end{remark}

The proof is based on an intrinsic decomposition of the associated set of multi-indices, which is similar to that in some spline construction such as \cite{chui1990multivariate}. However, as it will be seen below, this intrinsic decomposition will be used in a completely different way herein. To this end, a refinement of such a decomposition will be proposed, together with some basic properties. This will be discussed in \Cref{sec:intrinsic-decomposition}. Based on the decomposition, two sets of degrees of freedom are constructed in \Cref{sec:dof}. To help the readers get familiar with the notation and the main result, two-dimensional and three-dimensional examples are displayed in \Cref{sec:example}. The proof of unisolvency and continuity is given in \Cref{sec:proof}. 

The rest of the paper discusses some possible generalizations. In \Cref{sec:l2space} some discontinuous finite element spaces are constructed. \Cref{sec:Stokes} shows that the constructed finite element spaces can be used to establish some new two dimensional finite element Stokes complex. While in \Cref{sec:Hdiv}, another finite element smooth de Rham complex is built via constructing a generalized Stenberg element.

\subsection{Notation}
Some conventional notation is summarized here:  $\mathbb{I}_d$ denotes the set $\{0,1,\cdots, d\}$, and $\mathcal T = \mathcal T(\Omega)$ denotes a $d$ dimensional simplicial triangulation (a conforming triangulation) of domain $\Omega$ which can be exactly covered by simplices, $\bm{x}$ denotes a vertex of $\mathcal T$, $K$ denotes an element of $\mathcal T$ with vertices $\bm{x}_0,\cdots,\bm{x}_d$, and $\lambda_i$ denotes the barycenter coordinate associated to vertex $\bm{x}_i$ of $K$, $i = 0,\cdots,d$. When $d = 0$, i.e. $K$ is a vertex $\bm{x}$, define $\lambda_{\bm{x}} = 1$.

Given a subset $\bm I$ of $\mathbb{I}_d$, let $\langle \bm{I} \rangle$ denote the simplex taking vertices $\{ \bm x_{i}: i \in \bm I \}$ as its vertices. Equivalently, $\langle \bm I \rangle$ is the convex hull of $\{\bm x_i~:~ i \in \bm I\}$. Clearly, the mapping $\bm I \mapsto \lrangle{\bm{I}} := \operatorname{conv}(\{\bm x_i:i \in \bm I\})$ defines a bijection between all sub-simplices of $K$ and all nonempty subsets of $\mathbb{I}_d$.

\subsection{Argyris element}
\label{sec:argyris}
To gain some intuition and make the illustration smoother in the following, it is helpful to recall the triangular Argyris element \cite{argyris1968tuba} here. Given $u \in \Pcal_5(K)$, the degrees of freedom in notation of \Cref{thm:informal} are given as follows:

\begin{itemize}
\item[-]  The function value, first and second order derivatives of $u$ at each vertex $\bm x$ of element $K$.\vspace{1em}

For this set of degrees of freedom, the integer $n$ in \Cref{thm:informal} takes 0, 1, and 2. The corresponding bubble function spaces from \Cref{thm:informal} are $ \mathcal{B}_{\bm x, 0, 5} = \Span\{\lambda_{\bm x}^5\}, \mathcal{B}_{\bm x, 1, 5} = \Span\{\lambda_{\bm x}^4\},  \mathcal{B}_{\bm x, 2, 5} = \Span\{\lambda_{\bm x}^3\},$ respectively, and will be checked below. Recall that here $\lambda_{\bm x} = 1$ is a function which is defined only at vertex $\bm x$.

Take $\mathcal B_{\bm x,0,5}$ as an example. Since vertex $\bm x$ is of codimension $s = 2$, by definition \eqref{eq:Bfnk-informal}, the bubble space is spanned by $\lambda_{\bm x}^{p}$ for some nonnegative integer $p$. It follows from \eqref{eq:thm-informal-cond1} that $p = 5 - 0 $, while the second condition \eqref{eq:thm-informal-cond2} vacuously holds since $d = s = 2$. Therefore, the bubble space $\mathcal B_{\bm x,0,5} = \Span\{\lambda_{\bm x}^5\}$. Note again that $\lambda_{\bm x}$ is only defined at the vertex with value 1, and hence one can regard $\mathcal B_{\bm x,0,k}$ as the constant function space $\mathbb R$ for $k = 3,4,5$.

\item[-] 
The weighted moment
$$\frac{1}{|e|}\int_e \left(\frac{\partial u}{\partial \bm n}\right) \cdot(\lambda_{e,1}^2\lambda_{e,2}^2)$$
for each edge $e$ of element $K$. Here $\lambda_{e,1}$ and $\lambda_{e,2}$ are the barycenter coordinates of edge $e$. \vspace{1em}

For this set of degrees of freedom, note that $n \le 1$. It suffices to show that $\mathcal B_{e,0,5}$ is an empty set, and  $\mathcal B_{e,1,5} = \Span\{\lambda_{e,1}^2\lambda_{e,2}^2\}$ is a one-dimensional polynomial space.

In fact, for $\mathcal B_{e,0,5}$, by definition it is spanned by $\lambda_{e,1}^p \lambda_{e,2}^q$, and the conditions \eqref{eq:thm-informal-cond1} and \eqref{eq:thm-informal-cond2} are specified as $p + q = 5 - 0 = 5$ and $p, q > 2^{1+1-1} - 0 = 2$, respectively. However, for integers $p$ and $q$, these two conditions contradict with each other, which implies that the bubble space $\mathcal B_{e,0,5}$ is an empty set.

For $\mathcal B_{e,1,5}$, the conditions \eqref{eq:thm-informal-cond1} and \eqref{eq:thm-informal-cond2} are specified as $p + q = 5 - 1 = 4$, and $p , q > 2^{1 + 1 - 1} - 1  = 1$, respectively. Only the pair $(p,q) = (2,2)$ meets the requirements. Therefore, $\mathcal B_{e,1,5} = \Span\{\lambda_{e,1}^2\lambda_{e,2}^2\}$.

\item[-] In each element $K$, the bubble space is spanned by $\lambda_{F,0}^{\sigma_0}\lambda_{F,1}^{\sigma_1}\lambda_{F,2}^{\sigma_2}$ where
$\sigma_0+\sigma_1 +\sigma_2 = 5$ and $\sigma_k > 2^{1+0-1} = 1$ for $k = 0,1,2$. Since no integers $\sigma_0, \sigma_1$, and $\sigma_2$ meet these requirements, the bubble space $\mathcal B_{K,0,5}$ is an empty set. 
\end{itemize}

\begin{figure}[htbp]
    \centering
    \includegraphics[]{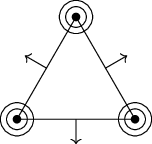}
    \caption{The illustration of degrees of freedom of the Argyris element.}
    \label{fig:example-arg}
\end{figure}
\Cref{fig:example-arg} illustrates the degrees of freedom of the Argyris element. It is well known that the Argyris element has $C^1$ continuity, see \cite{argyris1968tuba,brenner2008mathematical}. It is stressed that the above degrees of freedom are modified from those of the original Argyris element \cite{argyris1968tuba}, as on each edge $e$ of $K$, the value of $\frac{\partial u}{\partial \bm n}$ at the midpoint of $e$ is replaced by a weighted norm, as mentioned above.

\section{Intrinsic Decomposition}
\label{sec:intrinsic-decomposition}
Given a positive integer $k$, this section constructs a decomposition, called an \textit{intrinsic decomposition}, of the set of multi-indices
\begin{equation}
    \Sigma(\mathbb I_d, k) := \{(\alpha_0, \alpha_1, \cdots, \alpha_d) \in \mathbb{N}_0^{d+1} :~ \sum_{i = 0}^{d} \alpha_i = k\},
\end{equation}
which builds a relationship between a set of the geometric components of simplex $K$ and the set of all the multi-indices of degree $k$. 

\subsection{Definition and Assumption}

Given a nonempty index set $\bm I := \{i_0, i_1, \cdots, i_{d'}\} \subseteq \mathbb{N}_0$ and an integer $k \ge 0$, define
\begin{equation}
    \label{eq-sigmaIk}
    \Sigma(\bm{I}, k) := \{(\alpha_{i_0}, \alpha_{i_1}, \cdots, \alpha_{i_{d'}}) \in \mathbb{N}_0^{d'+1}:~ \sum_{i \in \bm{I}} \alpha_i = k\}.
\end{equation}

The decomposition is based on a \textit{continuity vector}  $\bm{r} := (r_1, \cdots, r_d)$ with $r_i$ a nonnegative integer, $i = 1, 2, \cdots, d$. The vector specifies the continuity of piecewise polynomial finite element functions across the internal subsimplices of the given conforming simplicial triangulation $\mathcal T(\Omega)$. For example, in the standard $C^r$ Bramble--Zl\'{a}mal element \cite{bramble1970triangular}, the finite element functions of piecewise polynomials are of $C^{2r}$ continuity across the vertices ($0$-simplices), and are of $C^{r}$ continuity across the internal edges ($1$-simplices) of the simplicial triangulation $\mathcal T(\Omega)$. For the more general case in two dimensions, the continuity vector is taken as $\bm{r} = (r_1, r_2)$. For $d$ dimensions, the component $r_s$ of the continuity vector represents the continuity of the finite element functions of piecewise polynomials when crossing $(d-s)$-dimensional simplices (equivalently, with the codimension $s$). For example, for the standard triangular Argyris element, the continuity vector is chosen as $\bm r = (r_1,r_2) = (1,2)$, since it admits $C^1$ continuity when crossing internal edges while $C^2$ continuity when crossing vertices, by the choice of degrees of freedom. While for the three dimensional \v Zen\'i\v sek element, the continuity vector is chosen as $\bm r = (r_1,r_2,r_3) =  (1,2,4)$.

Throughout this paper, the following assumption is required for the continuity vector $\bm r$, as well as the polynomial degree $k$, unless otherwise specified. Note that the assumption is a sufficient condition for the construction, which seems to appear naturally in the existing attempts in two and three dimensions, while it is still not clear whether it is a necessary condition for the existence of $C^r$ conforming finite element spaces, even for the two-dimensional case.

\begin{assumption}
    \label{assu}
    For the continuity vector $\bm{r} = (r_1, \cdots, r_d)$ with nonnegative integers $r_1, \cdots, r_d$ and the polynomial degree $k$, it holds 
    \begin{equation*}
        r_d \ge 2r_{d-1} \ge 4r_{d-2} \ge \cdots \ge 2^{d-1}r_1
    \end{equation*}
    and
    \begin{equation*}k \ge 2r_d+1.\end{equation*}
\end{assumption}

Given a continuity vector $\bm r$ and a polynomial degree $k$, the intrinsic decomposition is defined recursively as follows. For some technical reason, $r_0 = 0$ is always additionally assumed.

\begin{definition}[An intrinsic decomposition for $\Sigma(\mathbb I_d,k)$]
    \label{defn:decomp}
    Given a continuity vector $\bm{r} = (r_1, \cdots, r_d)$ and a polynomial degree $k$ satisfying \Cref{assu}, a decomposition of $\Sigma(\mathbb I_d,k)$ is defined inductively as follows:
    \begin{equation}
        \begin{multlined}
            \Sigma_d(\mathbb I_d,k) := \{(\alpha_0,\cdots, \alpha_d) \in \Sigma(\mathbb I_d,k) :~ \textrm{ There exists a subset } \bm{N}_d \subseteq \mathbb{I}_d, \\ \textrm{such that }\card(\bm{N}_d) = d
            \textrm{ and } \sum_{i \in \bm{N}_d} \alpha_i
            \le r_d \},
        \end{multlined}
    \end{equation}
    and then,
    \begin{equation}
        \label{eq:defn-sigmas}
        \begin{multlined}
            \Sigma_s(\mathbb I_d,k) := \{(\alpha_0, \cdots, \alpha_d) \in \Sigma(\mathbb I_d,k) :~ \textrm{ There exists a subset } \bm{N}_s \subseteq \mathbb{I}_d, \\ \textrm{such that }\card(\bm{N}_s) = s
            \textrm{ and } \sum_{i \in \bm{N}_s} \alpha_i \le r_s \}~\setminus ~\bigcup_{s' = s+1}^d\Sigma_{s'}(\mathbb I_d,k),
        \end{multlined}
    \end{equation}
    for $s = d-1, \cdots, 1$ sequentially. Finally, set
    \begin{equation}
        \begin{split}
            \Sigma_0(\mathbb I_d,k) & := \Sigma(\mathbb I_d,k)~\setminus~\bigcup_{s'=1}^{d}\Sigma_{s'}(\mathbb I_d,k).        \end{split}
    \end{equation}
\end{definition}

It follows from the definition of $\Sigma_s(\mathbb I_d,k), s = 0,1,\cdots, d$, that
\begin{equation}
    \label{eq:decomp}
    \Sigma(\mathbb I_d,k) = \Sigma_0(\mathbb I_d,k) \cup \Sigma_1(\mathbb I_d,k) \cup \cdots \cup \Sigma_d(\mathbb I_d,k),
\end{equation}
and that any two $\Sigma_s(\mathbb I_d,k)$ and $\Sigma_{s'}(\mathbb I_d,k)$ are disjoint if $s \neq s'$. 
In fact, as can be seen below, a further refined decomposition is needed for the construction and analysis of $C^r$ conforming finite element methods.

It is worth noting that in the definition of \eqref{eq:defn-sigmas}, the latter set on the right hand side in general is not a subset of the former set. To see this, consider a concrete two-dimensional example as follows, (here $\bm r = (r_1,r_2)$)

        \begin{align*}
            \Sigma_2(\mathbb I_2, k) & := \big\{(\alpha_0, \alpha_1, \alpha_2) \in \Sigma(\mathbb I_2, k):~ \alpha_1 + \alpha_2 \le r_2 \textrm{ or } \alpha_2 + \alpha_0 \le r_2 \textrm{ or } \alpha_0 + \alpha_1 \le r_2 \big\},     \\
            \Sigma_1(\mathbb I_2, k) & := \big\{(\alpha_0, \alpha_1, \alpha_2) \in \Sigma(\mathbb I_2, k):~ \alpha_0 \le r_1 \textrm{ or } \alpha_1 \le r_1 \textrm{ or } \alpha_2 \le r_1 \big\}\setminus \Sigma_2(\mathbb I_2, k), \\
            \Sigma_0(\mathbb I_2, k) & := \Sigma(\mathbb I_2, k)\setminus \left(\Sigma_2(\mathbb I_2, k)\cup \Sigma_1(\mathbb I_2, k)\right).
        \end{align*}

For example, with the continuity vector $\bm r = (r_1,r_2) = (1,4)$ and $k = 9$, the multi-index $(2,2,5)$ is in the set $\Sigma_2(\mathbb I_2,9)$, but none of the components is less than 1, that is, it does not belong to the set $$\big\{(\alpha_0, \alpha_1, \alpha_2) \in \Sigma(\mathbb I_2, k):~ \alpha_0 \le r_1 \textrm{ or } \alpha_1 \le r_1 \textrm{ or } \alpha_2 \le r_1 \big\}.$$

\begin{remark}
    \label{rmk:decomp}
    Given a continuity vector $\bm{q} = (q_1, \cdots, q_{d'})$ and a polynomial degree $k$ satisfying \Cref{assu}, for a given index set $\bm{I} = \{i_0, i_1, \cdots, i_{d'}\} \subseteq \mathbb{N}_0$, the decomposition of the set $\Sigma(\bm{I}, k)$ defined in \eqref{eq-sigmaIk} can be defined in a similar way as \Cref{defn:decomp}.
\end{remark}

\begin{remark}
    The last set $\Sigma_0(\mathbb I_d,k)$ of  the intrinsic decomposition can be characterized as:
    \begin{equation*}\begin{split} \Sigma_0(\mathbb I_d,k)  = \big\{(\alpha_0,\cdots,\alpha_d) \in \Sigma(\mathbb I_d,k):~  \alpha_{i_1}+\cdots+\alpha_{i_s} > r_s,~~\forall \{i_1,\cdots,i_s\} \subsetneq \mathbb{I}_d, s = 1,2,\cdots, d\big\}. \end{split}\end{equation*}
\end{remark}

\begin{example}
For the Argyris element, the continuity vector $\bm r = (r_1 ,r_2) = (1,2)$, the polynomial degree $k = 5$, and the intrinsic decomposition defined above reads 
$$\Sigma_0(\mathbb I_2,5) = \emptyset,$$
$$\Sigma_1(\mathbb I_2,5) = \{(1,2,2), (2,1,2),(2,2,1)\},$$
and 
$$\Sigma_2(\mathbb I_2,5) = \Sigma(\mathbb I_2, 5) \setminus \big(\Sigma_0(\mathbb I_2,5)  \cup \Sigma_1(\mathbb I_2,5) \big) .$$

Note that the set $\Sigma_0(\mathbb I_2,5)$ {is used to define} the degrees of freedom inside each element $K$ (thus no degrees of freedom are assigned inside element $K$), while the set $\Sigma_1(\mathbb I_2,5)$ indicates the degrees of freedom on edges (thus three degrees of freedom are assigned on edges in total).

\end{example}

\subsection{A refined intrinsic decomposition}

Recall the main result in \Cref{thm:informal}, the bubble function space inside the $d$-dimensional simplex element $K$ is defined by using the index set $\Sigma_0(\mathbb I_d,k)$, with respect to the continuity vector $\bm r = (r_1,\cdots,r_d) = (r,2r,\cdots, 2^{d-1}r)$. This subsection is to define the (in)complete bubble function spaces inside the $(d-s)$-dimensional simplices of $K$ by using the corresponding subsets $\Sigma_s(\mathbb I_d,k)$ of $\Sigma(\mathbb I_d,k)$ for $1\le s \le d$. To this end, a refined intrinsic decomposition of $\Sigma_{s}(\mathbb I_d,k)$ will be further introduced. As a result, each subset of the refined intrinsic decomposition is assigned to a $(d-s)$-dimensional simplex of $K$, indexed as a pair $(\bm N,n)$.

\begin{definition}
    \label{defi-NDelta}
    Let a continuity vector $\bm{r} = (r_1, \cdots, r_d)$ and a polynomial degree $k$ satisfying \Cref{assu} be given. For a given multi-index $\alpha \in \Sigma_{s}(\mathbb I_d,k)$ from \Cref{defn:decomp}, let $\bm{N}(\alpha)$ denote the subset defined in \eqref{eq:defn-sigmas}, that is
    $$\bm{N}_s \subseteq \mathbb{I}_d \text{  such that } \sum_{i \in \bm{N}_s}\alpha_i \le r_s  \text{ and  }\card(\bm{N}_s) = s, $$ choose one if there are multiple possible choices. Notice that here the subscript $s$ is only used to emphasize the cardinality of the set. Then let $$\Delta(\alpha) := \mathbb{I}_d\setminus \bm{N}(\alpha)$$ and $$n(\alpha) := \sum_{i \in \bm{N}(\alpha)}\alpha_i.$$

\end{definition}

Particularly, when $\bm{N}(\alpha) = \emptyset$, let $n(\alpha) = 0$. To adopt this, let $r_0 = 0$ in this section for convenience. Notice that this definition does not conflict with \Cref{assu} since only $r_1,\cdots ,r_d$ are restricted therein.

\begin{remark}
    \label{rmk:Delta-n}
    As it will be seen below, $\bm N(\alpha), \Delta(\alpha)$ and $n(\alpha)$ will be used to define the degree of freedom associated with multi-index $\alpha$. In particular, the set $\bm N(\alpha)$ will be used to define the normal vector(s) involved in the degree of freedom $\varphi_{\alpha}(\cdot)$ of \eqref{eq:fe-dof-local} below, while the set $\Delta(\alpha)$ will be used to locate the associated subsimplex of $\varphi_{\alpha}(\cdot)$. At last, $n(\alpha)$ is the order of the normal derivative taken in $\varphi_{\alpha}(\cdot)$.
\end{remark}

\begin{remark}
    \label{rmk:n-alpha}
    If $\alpha \in \Sigma_s(\mathbb I_d,k)$, under \Cref{assu}, then it holds that $n(\alpha) := \sum_{i\in \bm{N}(\alpha)}\alpha_i \le r_s$ and $\alpha_j > r_{s+1} - n(\alpha) \ge r_s$ for any $j \in \Delta(\alpha).$
\end{remark}

At a first glance, it seems that the choice of $\bm N(\alpha)$ might be too arbitrary to make things in order. Nevertheless, it can be proved that the choice is always unique provided that \Cref{assu} holds.

\begin{proposition}[Uniqueness of $\bm N$ and $\Delta$]
    \label{prop:uni-N-Delta}
    Under \Cref{assu}, for any multi-index $\alpha \in \Sigma_s(\mathbb I_d, k) \subseteq \Sigma(\mathbb I_d,k)$, there exists a unique subset $\bm{N}(\alpha)$ and consequently a unique $\Delta(\alpha) := \mathbb{I}_d \setminus \bm{N}(\alpha)$ such that $$\sum_{i \in \bm{N}(\alpha)}\alpha_i \le r_s \text{ and }\card(\bm{N}(\alpha)) = s.$$
\end{proposition}
\begin{proof}
    Suppose that, there are two different sets $\bm{N}_1$ and $\bm{N}_2$, with $\card(\bm{N}_1) = \card(\bm{N}_2) = s$ and
    \begin{equation*}
        \sum_{i \in \bm{N}_1} \alpha_i \le r_s, \quad \sum_{i \in \bm{N}_2} \alpha_i \le r_s.
    \end{equation*}

    If $s = d$, it holds that $\mathbb{I}_d = \bm N_1 \cup \bm N_2$, hence
    \begin{equation*}
        k = \sum_{i \in \mathbb I_d} \alpha_i \le \sum_{i \in \bm{N}_1} \alpha_i + \sum_{i \in\bm{N}_2} \alpha_i \le r_d + r_d < k,
    \end{equation*}
    which is a contradiction.

    If $s < d$, it holds that $\card(\bm N_1 \cup \bm N_2) \ge s+1$, hence
    \begin{equation*}
        \sum_{i \in \bm{N}_1 \cup \bm{N}_2} \alpha_i \le \sum_{i\in \bm{N}_1} \alpha_i + \sum_{i \in \bm{N}_2} \alpha_i \le r_s + r_s \le r_{\card(\bm N_1 \cup \bm N_2)},
    \end{equation*}
    which contradicts with the choice of $\bm{N}_i$ and $s$. This completes the proof.
\end{proof}

A refined version of the above intrinsic decomposition $\Sigma_0(\mathbb{I}_d, k), \Sigma_1(\mathbb{I}_d, k), \cdots, \Sigma_d(\mathbb{I}_d, k)$ of the set $\Sigma(\mathbb{I}_d, k)$ can be established as follows. Such a refined version is based on proper subsets $\bm{N}$ of $\mathbb{I}_d$ and nonnegative integers $n \in \mathbb{N}_0$, which is a natural consequence of \Cref{rmk:Delta-n}.

\begin{definition}[A refined intrinsic decomposition of $\Sigma(\mathbb I_d,k)$]
    \label{defi:SigmaNn}
    Given a continuity vector $\bm{r}$ and a polynomial degree $k$ with \Cref{assu}, for a proper subset $\bm{N} \subsetneq \mathbb{I}_d$ and an integer $n \in \mathbb{N}_0$, a refined intrinsic decomposition with respect to $\bm r$ and $k$ is defined as
    \begin{equation}
        \Sigma_{\bm{N},n}(\mathbb I_d,k) := \{\alpha \in \Sigma(\mathbb I_d,k): \bm{N}(\alpha) = \bm{N} \textrm{ and } n(\alpha) = n\}.
    \end{equation}
    Here, given $\alpha \in \Sigma(\mathbb{I}_d, k)$, $\bm{N}(\alpha)$ and $n(\alpha)$ are defined in \Cref{defi-NDelta} above. Running over all $(\bm N,n)$ such that $\Sigma_{\bm N,n}$ is not empty leads to the following refined decomposition (a disjoint union)
    \begin{equation} \label{eq:decomp-refine}
        \Sigma(\mathbb I_d,k) = \bigcup_{\bm{N},n} \Sigma_{\bm{N},n}(\mathbb I_d,k).
    \end{equation}

    Given a nonempty proper subset $\bm{N} \subsetneq \mathbb{I}_d$, let $n\le r_{\card(\bm{N})}$ and the multi-index $\theta \in \Sigma(\bm{N},n)$ with the set $\Sigma(\bm{N}, n)$ defined in \eqref{eq-sigmaIk}, define
    \begin{equation}
        \label{eq:Sigma-Nn-theta}
        \Sigma_{\bm{N},n,\theta}(\mathbb I_d,k) = \{\alpha \in \Sigma_{\bm{N},n}(\mathbb I_d,k) :~ \alpha_i = \theta_i, ~\forall i \in \bm{N} \}.
    \end{equation}

\end{definition}

\begin{example}
Consider the refined intrinsic decomposition for the Argyris element, i.e., the refined intrinsic decomposition of $\Sigma(\mathbb I_2,5)$ with respect to the continuity vector $\bm r = (1,2)$. Here the focus is put in the classification of degrees of freedom at vertices and on edges, see \Cref{fig:decomp-arg} for an illustration for the refined intrinsic decomposition of the set $\Sigma(\mathbb I_2, 5)$.
\begin{itemize}
    \item[-]
Consider $\bm x_0$ as a typical vertex. In this case, $\bm N$ will be taken as $\{1,2\}$, and $n$ can be $0,1,2 = r_2$. It follows that
$$ \Sigma_{\{1,2\},0}(\mathbb I_2,5) = \{(5,0,0)\}, \quad  \Sigma_{\{1,2\},1}(\mathbb I_2,5) = \{(4,1,0), (4,0,1)\},$$
and
$$\Sigma_{\{1,2\},2}(\mathbb I_2,5) = \{(3,0,2), (3,1,1), (3,2,0)\},$$
{corresponding to} the zeroth, first and second order derivatives at vertex $\bm x_0$, {respectively}.
\item[-]
Consider $e_0= \langle \bm x_1, \bm x_2\rangle$ as a typical edge. In this case, $\bm N$ will be taken as $\{0\}$, and the possible $n$ can be $0,1 = r_1$. It follows that 
$$\Sigma_{\{0\},0}(\mathbb I_2,5) = \emptyset, \quad  \Sigma_{\{0\},1}(\mathbb I_2,5) = \{(1,2,2)\},$$
which is corresponding to the degrees of freedom on edge $e_0$. That is, no degree of freedom is for the function value, while one degree of freedom is for the first-order normal derivative on each edge.
\end{itemize}
\end{example}

\begin{figure}[htbp]
    \centering
    \includegraphics[width = 12.5cm]{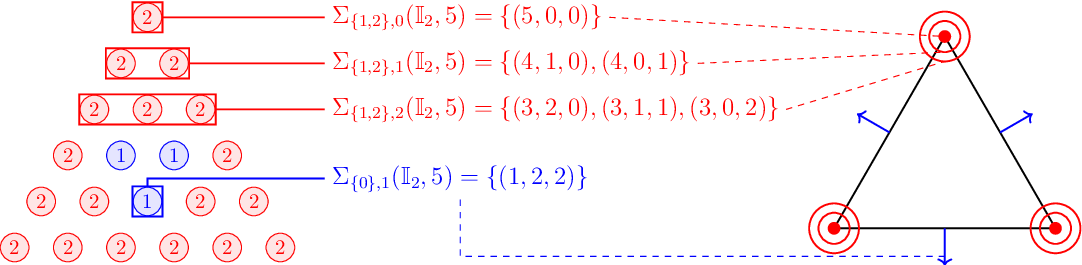}
    \caption{The refined intrinsic decomposition of $\Sigma(\mathbb I_2, 5)$.}
    \label{fig:decomp-arg}
\end{figure}

From \Cref{defi-NDelta} and \Cref{prop:uni-N-Delta}, the following properties hold about this refined decomposition.
\begin{proposition}
    \label{prop:SigmaNn}
    Given a continuity vector $\bm{r} = (r_1,\cdots, r_d)$ and a polynomial degree $k$ satisfying \Cref{assu}, the refined intrinsic decomposition \eqref{eq:decomp-refine} possesses the following properties:
    \begin{enumerate}
        \item These sets $\Sigma_{\bm N,n}(\mathbb I_d,k)$ are disjoint, i.e., $$\Sigma_{\bm{N}, n}(\mathbb{I}_d, k)\cap\Sigma_{\bm{N}', n'}(\mathbb{I}_d, k) = \emptyset\text{ if }(\bm{N},n) \neq  (\bm{N}',n'). $$
        \item For any nonnegative integer $n$, it holds that $$\Sigma_{\bm{N}, n}(\mathbb{I}_d, k)\subseteq \Sigma_{\card(\bm{N})}(\mathbb{I}_d, k). $$
        \item For any pair $(\bm{N}, n)$ such that $n > r_{\card(\bm{N})}$, it holds that $$\Sigma_{\bm{N}, n}(\mathbb{I}_d, k) = \emptyset.$$
    \end{enumerate}

\end{proposition}

From \Cref{prop:SigmaNn}, the refined decomposition \eqref{eq:decomp-refine} of $\Sigma(\mathbb{I}_d, k)$ can be written in more details as
$$\Sigma(\mathbb I_d,k) = \bigcup_{\substack{\bm{N}\subsetneq \mathbb{I}_d \\ n \le r_{\card({\bm{N}})}}} \Sigma_{\bm{N},n}(\mathbb I_d,k),$$
where the range of pairs $(\bm N, n)$ is clarified.

Similar refined decompositions of $\Sigma(\mathbb I_d,k)$ were proposed in the literature, cf., \cite{chui1990multivariate, alfeld1992dimension}, which are used to determine the basis functions of the super spline space. However, no $C^r$ finite element methods were constructed so far.

\begin{remark} In \cite{chui1990multivariate}, the notion $N_{ji}$ therein is a subset of $\Sigma_{d - j}(\mathbb{I}_d, k)$ of this paper, which needs a further decomposition herein.
\end{remark}

It now returns to the construction of degrees of freedom. \Cref{prop:bijection} below, based on the refined intrinsic decomposition in $d$ dimensions, builds a bijection mapping. Such a mapping will be used to define the local degrees of freedom inside subsimplices of element $K$. Consequently, a unified definition of degrees of freedom can be carried out, which is a local version of degrees of freedom. The locality and unification provide many benefits when proving the unisolvency and continuity.
Before the introduction of the bijection mapping, an intrinsic decomposition based on a continuity vector $\bm q = (q_1, \cdots, q_{d'})$ other than $\bm r$ is needed for the set of multi-indices associated to the pair $(\bm I, k')$ where $\bm I := \{i_0, i_1, \cdots, i_{d'}\} \subseteq \mathbb I_d$ and $k' \in \mathbb N_0$. In particular, the last set of this decomposition, denoted as $\Sigma_0^{(\bm q)}(\bm I,k')$, reads
\begin{equation}\label{eq:sigma0q}\begin{split} \Sigma_0^{(\bm q)}(\bm I,k')  =& \big\{(\alpha_{i_0},\cdots,\alpha_{i_{d'}}) \in \Sigma(\bm I,k'):~ \\ &  \alpha_{j_1}+\cdots+\alpha_{j_s} >  q_s,~~\forall \{j_1,\cdots,j_s\} \subsetneq \bm I, s = 1,2,\cdots, d' \big\}. \end{split}\end{equation}

\begin{remark}
    For the case $\card(\bm{I}) = 1$, $\bm{q}$ is an empty continuity vector, and $\Sigma_0^{(\bm{q})}(\bm{I}, k') = \Sigma(\bm{I}, k')$.
\end{remark}

\begin{proposition}
    \label{prop:bijection}
    Given a continuity vector $\bm{r} = (r_1, \cdots, r_d)$ and polynomial degree $k$ with \Cref{assu}, for a nonempty proper subset $\Delta \subsetneq \mathbb{I}_d$, let $\bm{N} := \mathbb{I}_d \setminus \Delta$, $s := \card(\bm{N})$, $n\le r_s$. Then it holds that the mapping
    $$\mathcal{R}_{\bm{N}, \Delta}: \alpha \mapsto (\theta, \sigma) ,$$
    defined by $$\theta_i = \alpha_i,  i\in \bm{N} \text{ and }\sigma_i = \alpha_i, i\in \Delta,$$ is a bijection between $\Sigma_{\bm N, n}(\mathbb I_d,k)$ and $\Sigma(\bm N, n) \times \Sigma_0^{(\bm q)}(\Delta, k - n)$ with $\bm q = \bm q^{s, n} := (r_{s+1} - n, \cdots, r_{d} - n)$.
\end{proposition}

Before proving it, several examples are introduced and discussed, in order to clarify the statement of \Cref{prop:bijection} and its consequence in the main construction (will be shown in \Cref{sec:dof}) of this paper.

\begin{example}
Set $\mathbb I_d = \mathbb I_2 = \{0,1,2\}$, $k \ge 5$, and the continuity vector is chosen as $\bm r = (1,2)$. For $k = 5$, it corresponds to the Argyris element, which has been discussed in \Cref{sec:argyris}. Here, the bubble spaces and the corresponding degrees of freedom are reinterpreted in the language of the refined intrinsic decomposition. Notice that $\bm N$ considered in \Cref{prop:bijection} is a nonempty proper subset, all the possible cases are enumerated below.
\begin{itemize}
    \item[-]  $\card(\bm N) = 1$, and $n = 0$. Suppose that $\bm N = \{0\}$, then 
    $$ \Sigma_{\bm{N}, n}(\mathbb I_2, k) = \Sigma_{\{0\}, 0}(\mathbb I_2, k) := \{(0,p,q) : p+q = k, p, q > r_2 = 2\},$$
    where the restriction of $p$, $q$ comes from the definition of $\bm N$ (more precisely, the sum of any two indices is greater than $r_2 = 2$). In this case, $\Sigma(\bm N, n) = \Sigma(\{0\}, 0) = \{(0)\}$, $\Delta = \{1, 2\}$, $\bm q = (2)$, and it follows from the definition \eqref{eq:sigma0q} that 
    $$ \Sigma_{0}^{(\bm{q})}(\Delta, k - n) = \Sigma_{0}^{(2)}(\{1,2\}, k - 0) := \{(p,q): p+q = k, p, q > 2\}.$$ 
    In this case, it is easy to check the bijection holds. In particular, when $k = 5$, both $\Sigma_{\{0\},0}(\mathbb I_2,5)$ and $\Sigma_{0}^{(2)}(\{1,2\}, 5 - 0)$ are empty, which indicates that there are no degrees of freedom on edge with respect to the function value.
    \item[-] $\card(\bm N) = 1$, and $n = 1$. Suppose that $\bm N = \{0\}$, then 
    $$ \Sigma_{\bm{N}, n}(\mathbb I_2, k) = \Sigma_{\{0\},1}(\mathbb I_2, k) := \{(1,p,q): p+q = k-1, p, q > r_2 - 1 = 1\}.$$
    In this case, $\Sigma(\bm{N}, n) = \Sigma(\{0\},1) = \{(1)\}$, $\Delta = \{1, 2\}$, and $\bm q = (1)$. It follows from \eqref{eq:sigma0q} that
    $$ \Sigma_{0}^{(\bm{q})}(\Delta, k - n) = \Sigma_{0}^{(1)}(\{1,2\}, k - 1) := \{ (p,q): p+q = k-1, p,q > 1\}.$$
    The bijection also holds. In particular, when $k = 5$, all the sets only consist of a single component.
    \item[-] $\card(\bm N) = 2$ and $n \le r_2 = 2$. Suppose that $\bm N=\{1,2\}$ and $\Delta = \{0\}$, then 
    $$ \Sigma_{\bm{N}, n}(\mathbb I_2, k) = \Sigma_{\{1,2\},n}(\mathbb I_2,k) = \{(k-n,\alpha_1,\alpha_2): \alpha_1 + \alpha_2 = n\},$$
    and $\bm q$ is an empty continuity vector. Thus by definition \eqref{eq:sigma0q} 
    $ \Sigma_{0}^{(\bm{q})}(\Delta, k - n) = \Sigma(\{0\},k-n) = \{(k-n)\}$, $\Sigma(\{1,2\}, n) = \{(\alpha_1,\alpha_2), \alpha_1 + \alpha_2 = n\}$. In this case, the bijection also holds.
\end{itemize}
\end{example}
The above example implies that the bijection relationship shown in \Cref{prop:bijection} in two dimensions might be too simple to gain more information since either $\bm N$ or $\Delta$ may be reduced to a set with only one element. To this end, consider the following three-dimensional example, showing some non-triviality of such a relationship. For simplicity, only the statement itself of the proposition is checked in the following, and the corresponding degrees of freedom will be not expanded. A detailed illustration of three-dimensional finite element spaces will be displayed in \Cref{sec:3d-example}.

\begin{example}
\label{exa:decomp-3d}
Let $d= 3$, $k = 13$, and the continuity vector $\bm r = (1,3,6)$. Consider the case where $\bm N = \{0,1\}$ and $\Delta = \{2,3\}$. For $n = 2$ (hence $\bm{q} = (6 - 2) = (4)$), it holds that $\alpha_0 + \alpha_1 = 2$, which implies $(\alpha_0,\alpha_1) = (0,2), (1,1)$ or $(2,0)$. The restriction on $(\alpha_2,\alpha_3)$ then are $2 + \alpha_2, 2 + \alpha_3 > r_3 = 6$, and $2 + \alpha_2 + \alpha_3 = 13.$ A direct enumeration obtains that $\Sigma_{\bm{N}, n}(\mathbb I_3, k) = \Sigma_{\{0,1\},2}(\mathbb I_3, 13)$ is the union of the following three sets
$$ \{ (0,2, \alpha_2,\alpha_3) : \alpha_2 + \alpha_3 = 11,\alpha_2, \alpha_3 > 4\} := \{(0,2,5,6), (0,2,6,5)\},$$
$$ \{ (1,1, \alpha_2,\alpha_3) : \alpha_2 + \alpha_3 = 11,\alpha_2, \alpha_3 > 4\} := \{(1,1,5,6), (1,1,6,5)\},$$
$$ \{ (2,0, \alpha_2,\alpha_3) : \alpha_2 + \alpha_3 = 11,\alpha_2, \alpha_3 > 4\} := \{(2,0,5,6), (2,0,6,5)\}.$$

As a result, the set $\Sigma_{\{0,1\},2}(\mathbb I_3, 13)$ can be decomposed as
$$\Sigma_{\{0,1\},2}(\mathbb I_3,13) = \{(0,2),(1,1),(2,0)\} \times \{(5,6), (6,5)\}.$$
This is exactly what \Cref{prop:bijection} states, since 
$$\Sigma(\bm{N}, n) = \Sigma(\{0,1\}, 2) = \{(0,2),(1,1),(2,0)\},$$
and
$$\Sigma_0^{(\bm q)}(\Delta, k - n) = \Sigma_0^{(4)}(\{2,3\}, 11) = \{(5,6), (6,5)\}.$$

\end{example}

\begin{figure}[htbp]
    \centering
    \includegraphics[width = 6cm]{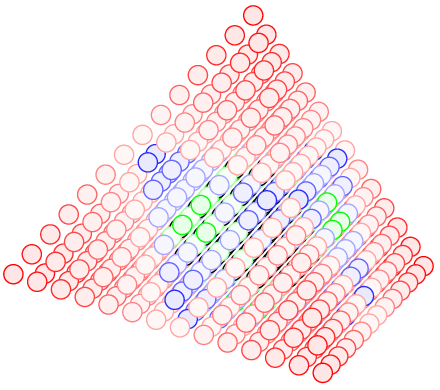} \qquad
    \includegraphics[width = 6cm]{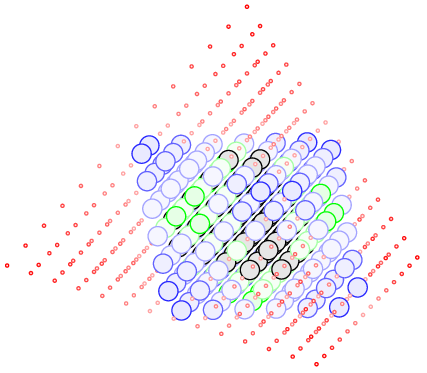} \\
    \includegraphics[width = 6cm]{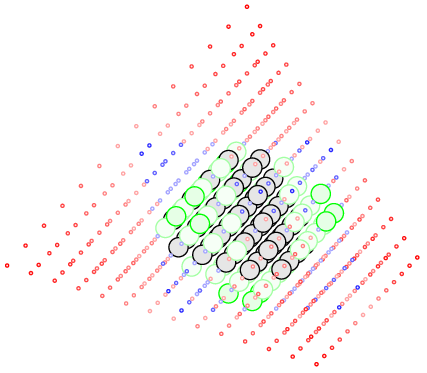} \qquad
    \includegraphics[width = 6cm]{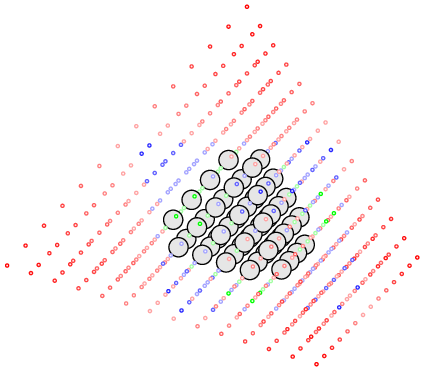}
    \caption{The refined intrinsic decomposition of $\Sigma(\mathbb I_3, 13)$ with $\bm{r} = (1, 3, 6)$.}
    \label{fig:decomp-3D}
\end{figure}

Finally, \Cref{prop:bijection} is proved to close this section. The basic argument in the proof is similar to that of \Cref{exa:decomp-3d}, with more technicality.

\begin{proof}[Proof of \Cref{prop:bijection}]
    For $\alpha \in \Sigma_{\bm{N},n}(\mathbb I_d,k)$ and $(\theta, \sigma) = \mathcal{R}_{\bm{N}, \Delta}(\alpha)$, it holds that
    \begin{equation*}
        \sum_{i\in \bm{N}} \theta_i = \sum_{i\in \bm{N}} \alpha_i = n, \quad \sum_{i\in \Delta} \sigma_i = \sum_{i\in \Delta} \alpha_i = k - n.
    \end{equation*}
    Moreover, for any nonempty subset $\bm{N}'$ of $\Delta$, since $\bm N' \cap \bm N = \emptyset$, it holds that
    \begin{equation*}
        \sum_{i \in \bm{N}'} \sigma_i = \sum_{i \in \bm{N}'} \alpha_i = \sum_{i \in \bm{N}\cup\bm{N}'} \alpha_i - \sum_{i \in \bm{N}} \alpha_i > r_{\card(\bm{N}') + s} - n = q_{\card(\bm{N}')}.
    \end{equation*}
    Hence $\theta$ belongs to $\Sigma(\bm{N},n)$ and $\sigma$ belongs to $\Sigma_0^{(\bm{q})}(\Delta, k - n)$.

    It is straightforward to see that $$\mathcal{R}_{\bm{N}, \Delta} :\Sigma_{\bm{N}, n}(\mathbb I_d,k) \longrightarrow  \Sigma(\bm{N}, n) \times \Sigma_0^{(\bm{q})}(\Delta, k - n) $$ is an injection. Now, it suffices to show that $\mathcal{R}_{\bm{N}, \Delta}$ is surjective.
    For $\theta \in \Sigma(\bm{N}, n)$ and $\sigma \in \Sigma_0^{(\bm{q})}(\Delta,k - n)$, let $\alpha \in \Sigma(\mathbb{I}_d, k)$ such that $\alpha_i = \theta_i$ for $i \in \bm{N}$ and $\alpha_i = \sigma_i$ for $i \in \Delta$. Then it holds that
    \begin{equation*}
        \sum_{i \in \bm{N}} \alpha_i = \sum_{i \in \bm{N}} \theta_i = n \le r_{s}.
    \end{equation*}
    It remains to show for any proper subset $\bm{N}'$ of $\mathbb{I}_d$ with $s + 1 \le \card(\bm{N}') \le d$, $\sum_{i \in \bm N} \alpha_i > r_{\card(\bm{N}')}$. If $\bm{N}$ is a subset of $\bm{N}'$, then it holds that $\card(\bm{N}'\setminus\bm{N}) = \card(\bm{N}') -  s$. It follows that
    \begin{equation*}
        \sum_{i \in \bm{N}'} \alpha_i = \sum_{i \in \bm{N}'\setminus\bm{N}} \sigma_i + \sum_{i \in \bm{N}} \theta_i > q_{\card(\bm{N}') - s} + n = r_{\card(\bm{N}')},
    \end{equation*}
    where the first equation is from the construction of $\alpha$, and  the last equation is from the definition of $\bm q$.
    For the other case, if $\bm{N}$ is not a subset of $\bm{N}'$, then $\card(\bm{N}'\setminus\bm{N}) \ge \card(\bm{N}') -  s + 1$. It follows that
    \begin{equation*}
        \sum_{i \in \bm{N}'} \alpha_i \ge \sum_{i \in \bm{N}'\setminus\bm{N}} \sigma_i > q_{\card(\bm{N}') - s + 1} = r_{\card(\bm{N}') + 1} - n \ge r_{\card(\bm{N}')}.
    \end{equation*}
    In conclusion, it holds that $\alpha \in \Sigma_{\bm{N}, n}(\mathbb{I}_d, k)$ and $\mathcal{R}_{\bm{N}, \Delta}(\alpha) = (\theta, \sigma)$, which proves the surjection part.

\end{proof}

In the following sections, \Cref{prop:bijection} will be frequently used. Simply speaking, the components in the set $\Sigma(\bm{N}, n)$ will correspond to certain (higher-order) derivatives that appear in the degrees of freedom, e.g. $u$, $\frac{\partial u}{\partial \bm n}$, $\frac{\partial^2 u}{\partial \bm n_1\partial \bm n_2}$, etc. The components in the set $\Sigma_{0}^{(\bm q)}(\Delta, k-n)$ will be used to define the bubble function space inside the subsimplex $\langle \Delta \rangle$. Till now, the relationship between the bubble function spaces and the corresponding components of the refined intrinsic decomposition has been established.

\section{Two sets of degrees of freedom: local and global}
\label{sec:dof}

This section introduces two sets of degrees of freedom, and shows their equivalence. The first  set of degrees of freedom \eqref{eq:fe-dof-local} will more intuitively portray the degrees of freedom on each subsimplex; while the second set of degrees of freedom \eqref{eq:theta-sigma} provides another perspective of \eqref{eq:fe-dof}, which makes sure that the proposed degrees of freedom can be used to design $C^r$ finite element spaces.

Given a subsimplex $F$, the averaging inner product $\frac{1}{\left\vert F \right\vert}\langle f, g \rangle_F$ is defined as $\frac{1}{\left\vert F \right\vert}\int_{F} f \cdot g $ if $F$ is not a vertex, as $f(F)\cdot g(F)$ if $F$ is a vertex. For the sake of clarity, the integral formulation $\frac{1}{\left\vert F \right\vert}\int_{F}f\cdot g$ will be also used when $F$ is a vertex, and should be understood as $f(F)g(F)$.

In what follows, the unit normal vectors of subsimplex $F$ of the element $K$ will be re-indexed for convenience. In particular, given a subsimplex $F := \lrangle{\Delta} := \operatorname{conv}\{\bm{x}_i :~ i\in \Delta\}$ of $K$, the orthonormal outer normal vectors will be denoted as $\bm{n}_{F, i}$, $i \in \bm{N} := \mathbb{I}_d\setminus \Delta$. A specific choice of these normal directions does not affect the construction and result, since the span of these normal vectors is the same space which is perpendicular to subsimplex $F$. However, the re-indexing can make the following proof more concise.
The readers might recall the definition of $\bm N(\alpha), n(\alpha), \Delta(\alpha)$ from Definition~\ref{defi-NDelta}.

\begin{definition}[Two sets of degrees of freedom]
    Given a continuity vector $\bm{r} = (r_1, \cdots, r_d)$ and a polynomial degree $k$ with \Cref{assu}, two sets of degrees of freedom are introduced as follows. Note that unless otherwise specified,  $\Delta$ and $\bm N$ are dependent on $\bm r$.
    \begin{enumerate}
        \item For $\alpha \in \Sigma(\mathbb I_d,k)$, let $\Delta = \Delta(\alpha)$ and $F := \lrangle{\Delta} $, let $\bm N = \bm N(\alpha)$ and $n = n(\alpha) := \sum_{i \in \bm{N}}{\alpha_i}$. For $u \in \mathcal P_k(K)$, define
              \begin{equation}
                  \label{eq:fe-dof-local}
                  \varphi_{\alpha} : u \longmapsto \frac{1}{\left\vert F\right\vert } \bigg \langle \frac{\partial^{n}}{ \prod_{i\in \bm{N}}\partial \bm{n}_{F, i}^{ \alpha_i}} u\big\vert _{F} , \bm{\lambda}^{\Delta}\alpha \bigg \rangle_{F},
              \end{equation}
              where $\bm \lambda^{\Delta}\alpha := \prod_{i  \in \Delta} \lambda_i^{\alpha_i}$. There is no distinction among $\lambda_i$, $\lambda_{K, i}$ and $\lambda_{F, i}$, since $\lambda_{K, i}\vert _F = \lambda_{F, i}$ for $F \subseteq K$.
        \item Given a nonempty proper subset $\Delta \subsetneq \mathbb{I}_d$ and $F := \lrangle{\Delta}$, let $\bm{N} := \mathbb{I}_d \setminus \Delta$, $s := \card(\bm{N})$, $n\le r_s$, and $\bm q = \bm q^{s, n}:= (r_{s+1} - n, \cdots, r_{d} - n)$. With $\theta \in \Sigma(\bm{N},n)$ and $\sigma \in \Sigma_0^{(\bm{q})}(\Delta, k-n)$ {defined in \eqref{eq:sigma0q}}, for $u \in \mathcal P_k(K)$, define
              \begin{equation}
                  \label{eq:theta-sigma}
                  \varphi_{\theta,\sigma}: u \longmapsto \frac{1}{\left\vert F\right\vert } \bigg \langle \frac{\partial^{n}}{ \prod_{i\in \bm{N}}\partial \bm{n}_{F, i}^{\theta_i}} u\big\vert _{F} , \bm{\lambda}^{\Delta}\sigma \bigg \rangle_{F},
              \end{equation}
              where $\bm \lambda^{\Delta}\sigma := \prod_{i \in \Delta} \lambda_i^{\sigma_i}$. Moreover, for $\Delta = \mathbb I_d$ and $\sigma \in \Sigma_0(\mathbb{I}_d, k)$ (which implies that $\bm{N}(\sigma)$ is empty), define $\varphi_{\emptyset, \sigma} := \varphi_{\sigma}$ as in \eqref{eq:fe-dof-local}.
              Here, for $\bm{r}$ and $k$ with \Cref{assu}, it is straightforward to show that both $\bm{q}^{s, n}$ and $k - n$ satisfy \Cref{assu}. Note that this is exactly the form in \Cref{thm:fe}.
    \end{enumerate}

\end{definition}

\begin{proposition}
    \label{prop:linear-comb}
    Under \Cref{assu}, for any nonempty proper subset $\Delta \subsetneq \mathbb{I}_d$, let $\bm{N} := \mathbb{I}_d \setminus \Delta$, $s := \card(\bm{N})$, $n\le r_s$, and $\bm q = \bm q^{s, n} := (r_{s+1} - n, \cdots, r_{d} - n)$. Then the following two sets of degrees of freedom
    \begin{equation}
        \label{eq:dof:varphi-alpha}
        \big\{ \varphi_\alpha :~ \alpha \in \Sigma_{\bm N, n}(\mathbb I_d,k) \big\}
    \end{equation}
    and
    \begin{equation}
        \big\{ \varphi_{\theta, \sigma} :~ \theta \in \Sigma(\bm N, n), \sigma \in \Sigma_0^{(\bm q)}(\Delta, k-n) \big\}
    \end{equation}
    coincide with each other. Moreover, the relationship can be written down explicitly, let $(\theta, \sigma) = \mathcal{R}_{\bm N(\alpha), \Delta(\alpha)} (\alpha)$, then it holds that $\varphi_\alpha = \varphi_{\theta, \sigma}$.
\end{proposition}

\begin{proof}
    From \Cref{prop:bijection}, the mapping $\mathcal{R}_{\bm{N},\Delta}: \alpha \mapsto (\theta, \sigma)$ is a bijection between $\Sigma_{\bm N, n}(\mathbb I_d,k)$ and $\Sigma(\bm N, n) \times \Sigma_0^{(\bm q)}(\Delta, k - n)$. The proposition {can be} immediately proved by this bijection.
\end{proof}

In what follows, a global version of degrees of freedom will be proposed.
Given a subsimplex $F$ with codimension $s$, it is clear that there are $s$ pairwise {orthonormal} normal vectors of $F$, denoted as $\bm n_{F, 0}, \cdots, \bm n_{F, s-1}$. Without loss of generality,
set $F = \langle \bm x_{s},\cdots,\bm x_{{d}}\rangle:= \operatorname{conv}(\{\bm x_s,\cdots, \bm x_d\})$ with codimension $s$. Define a bubble function space on $F$ (associated with the continuity vector  $\bm r$ and polynomial degree $k$) by

\begin{equation}
\label{eq:Bfnk}
\begin{split}
    \mathcal{B}_{F,n,k} & := \Span\big\{ \prod_{i = s}^d \lambda_{F, i}^{\sigma_{i}} ~:~ \sigma = {(\sigma_s,\cdots,\sigma_d)}\in \Sigma_0^{(\bm q)}(\mathbb I_{d} \setminus \mathbb I_{s-1}, k-n)\big\} \\ & = \Span\{ \lambda_{F, s}^{\sigma_{s}}\cdots \lambda_{F,d}^{\sigma_{{d}}} ~:~ \sigma = {(\sigma_s,\cdots,\sigma_d)}\in \Sigma_0^{(\bm q)}(\mathbb I_{d} \setminus \mathbb I_{s-1}, k-n)\}.
\end{split}
\end{equation}
Here $\lambda_{F, i}$ $:= \lambda_{K,i}|_F$, where $\lambda_{K,i}$ is the barycenter coordinate associated with vertex $\bm x_i$ of $K$, is also the barycenter coordinate associated to  vertex $\bm x_i$ with respect to $F$, $i = s,\cdots, d$, and the continuity vector
$\bm q = \bm q^{s, n} := (r_{s+1} - n, \cdots, r_{d} - n)$ for $n = 0, 1,\cdots, r_{s}$.

Given $u \in \Pcal_{k}(K)$, the degrees of freedom for the shape function space $\Pcal_k(K)$ are as follows:
\begin{equation}
    \label{eq:fe-dof}
    \frac{1}{\vert F\vert }\langle \bm D^\theta u, v \rangle_F \quad \forall v  \in \mathcal{B}_{F,n,k},
\end{equation}
for all the subsimplices $F$ of $d$-dimensional simplex $K$, and $ n = 0,1,\cdots, r_s$ (when $s = 0$, let $n = 0$).
Here $\bm D^{\theta} u$ represents an $n$-th order normal derivative of $u$ on $F$ when $n > 0$ ($\bm D^{\theta} u = u $ when $n = 0$), namely,
$$ \bm D^{\theta} u :=
    \frac{\partial^n}{\prod_{i = 0}^{s-1} \partial \bm n_{F, i}^{\theta_i}} u
$$
for some multi-index $\theta \in \Sigma(\mathbb{I}_{s - 1}, n)$, i.e., $n = \sum_{i = 0}^{s-1} \theta_i$. Then by the linearity of the space in \eqref{eq:Bfnk}, the degrees of freedom defined by \eqref{eq:fe-dof} is equivalent to \eqref{eq:theta-sigma}. Note that in fact there are $\dim \mathcal B_{F,n,k} = \card (\Sigma_{0}^{(\bm q)}(\mathbb I_d \setminus \mathbb I_{s-1}, k - n))$ degrees of freedom is defined by either \eqref{eq:fe-dof} or \eqref{eq:theta-sigma}.

The main result is the following unisolvency and $C^{r_1}$ continuity of the constructed finite element spaces. The following theorem makes everything in \Cref{thm:informal} precise.

\begin{theorem}
    \label{thm:fe}
    Given a continuity vector $\bm{r} = (r_1, \cdots, r_d)$ and a polynomial degree $k$ satisfying \Cref{assu}, the degrees of freedom defined in \eqref{eq:fe-dof} are unisolvent for $\mathcal{P}_k(K)$. Moreover, the global finite element space
    \begin{equation*}
        \begin{split}\{ u  \in L^2(\Omega):~u\vert _{K} & \in \mathcal P_k(K), u \text{ is single-valued for each degree of freedom}  \\ &\text{ on any subsimplex }F \text{ of dimension }\le d-1\}
        \end{split}
    \end{equation*}
    lies in $C^{r_1}(\Omega)$.

    Here $u$ is single-valued for each degree of freedom $\varphi_{\theta, \sigma}$ on $F$ means that, for any two $d-$dimensional simplices $K^+$ and $K^-$, sharing common subsimplex $F$, it holds that $\frac{1}{|F|}\langle D^{\theta}u|_{K^+}, v \rangle_{F} = \frac{1}{|F|}\langle D^{\theta}u|_{K^-},v \rangle_{F}$ for all $|\theta| = n$ and $v \in \mathcal B_{F,n,k}$.

\end{theorem}

The theorem is the main result of this paper, but the proof is complicated since the definition of the intrinsic decomposition and $\mathcal B_{F,n,k}$ is not straightforward in $d$ dimensions. Luckily, in two and three dimensions the characterization can be figured out, which is shown in \Cref{sec:example}. The theorem will be proved in \Cref{sec:proof} below.

\section{Examples in two and three dimensions}
\label{sec:example}

This section provides concrete examples in two and three dimensions. The following remark is also useful in the following, telling that when the bubble function spaces can be regarded as a complete polynomial bubble function space. 
Here completeness is a conventional mathematical notion. A polynomial bubble function space $\mathcal B$, defined on the subsimplex $F$ (with codimension $s$), is {\it complete} if there exist non-negative integers $r'$ and $k'$ such that $\mathcal B = (\lambda_{F,s}\cdots\lambda_{F,d})^{r'} \mathcal P_{k'}(F)$. On the contrary, $\mathcal B$ is incomplete if there do not exist such $r'$ and $k'$. 

    \begin{remark}
        For the continuity vector $\bm r$, if $r_s + 1 \le s(r_1+1)$ holds for $s = 1,\cdots, d$, then $$\Sigma_0(\mathbb I_d,k) = \{ (r_1 + 1 + \beta_0,\cdots, r_1 + 1 + \beta_d):\beta \in \Sigma(\mathbb{I}_d,k-(d+1)(r_1+1)) \}.$$
        As a result, it holds that 
        $$\mathcal{B}_{K,0,k} = \Span\{\lambda_{0}^{\sigma_0}\cdots \lambda_{d}^{\sigma_d}:~\sigma = {(\sigma_0,\cdots,\sigma_d)} \in \Sigma_0(\mathbb I_d,k)\} = (\lambda_0\cdots\lambda_d)^{r_1+1}\Pcal_{k-(d+1)(r_1+1)}.$$
    \end{remark}
    Under \Cref{assu}, the condition $r_s + 1 \le s(r_1+1)$ cannot hold in general. In fact, if $r_1 = 1$, then both $r_s \le 2s$ and $r_s \ge 2^{s-1}$ imply that $d \le 3$. Therefore, only for lower dimensional cases, the polynomial bubble function spaces can be possible to be complete. Such a situation will become more complicated for higher dimensions and the case of higher continuity. This, in some sense, explains the challenge of the construction of $C^r$ finite element spaces in any dimension.

\subsection{\texorpdfstring{$C^r$}{Cr} finite element spaces in two dimensions}
First, recall the two-dimensional Bramble--Zl\'{a}mal element \cite{bramble1970triangular}, which possesses $C^r$ continuity. The shape function space is taken as $\Pcal = \Pcal_k(K)$ of the space of polynomials of degree $\le k$ for $k \ge 4r+1$. Compared to the original paper \cite{bramble1970triangular}, a different but equivalent {set of} degrees of freedom is proposed herein. 

Given $u \in \Pcal_k(K)$, the degrees of freedom in the notation of \Cref{thm:informal} are as follows:
\begin{itemize}
\item[-] The function value, first, second, $\cdots$, $(2r)$-th order derivatives of $u$ at each vertex $\bm x $ of element $K$, the corresponding bubble function spaces are 
$$\mathcal B_{\bm x,n,k} = \Span\{\lambda_{\bm x}^{k-n}\} \text{ for } 0 \le n \le 2r,$$
where $\lambda_{\bm x}$ is the associated barycenter coordinate of vertex $\bm x$. This set of degrees of freedom is in fact defined by the set of multi-indices $\Sigma_2(\mathbb I_2,k)$.
\item[-] The weighted moment(s) 
$$
\frac{1}{|e|}\int_e \left(\frac{\partial^n u}{\partial \bm n^n} \right) \cdot v,\qquad \forall v \in\mathcal{B}_{e,n,k}$$
on each edge $e$ of element $K$ for $0 \le n \le r$. Here $\bm n$ is the unit normal vector of edge $e$, and the bubble function spaces $\mathcal B_{e,n,k}$ read

$$
\mathcal{B}_{e,n,k} = \Span\{\lambda_{e,1}^{\sigma_1}\lambda_{e,2}^{\sigma_2}\} = (\lambda_{e,1}\lambda_{e,2})^{2r+1-n} \Pcal_{k+n-2(2r+1)},$$
where $\sigma_1 + \sigma_2 = k - n$ and $\sigma_1, \sigma_2 \ge 2r+1 - n$.
This set of degrees of freedom is {defined by} the set of multi-indices $\Sigma_1(\mathbb I_2,k)$.
\item[-] The weighted moment(s)
$$ \frac{1}{\vert K\vert } \int_{K} u \cdot v \quad \forall v \in \mathcal{B}_{K,0,k} ={ \Span\{\lambda_{0}^{\sigma_0}\lambda_1^{\sigma_1}\lambda_{2}^{\sigma_2}:~\sigma \in \Sigma_0(\mathbb I_2,k)\}}.$$
This set of degrees of freedom is {defined by} the set of multi-indices $\Sigma_0(\mathbb I_2,k)$.
\end{itemize}

\Cref{fig:example-2d} illustrates the cases when $r = 1,k = 6$ and $r = 2, k = 9$. The general case in two dimensions can be obtained by replacing $(r,2r)$ by $(r_1,r_2)$, see \Cref{fig:decomp-2d}.

\begin{figure}[htbp]
    \centering
    \includegraphics[width = 4cm]{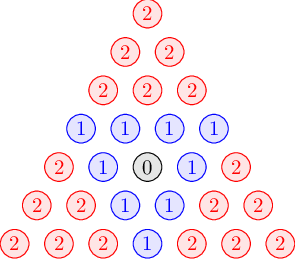} \qquad \quad \ 
    \includegraphics[width = 4cm]{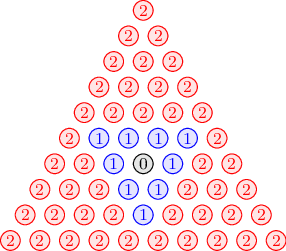} \\
    \vspace{5pt}
    \quad\includegraphics[]{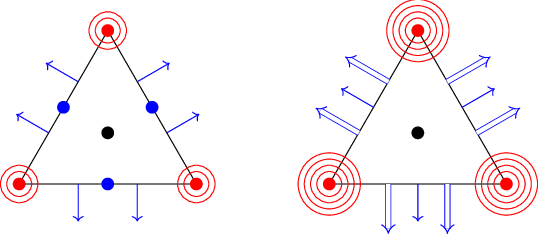}
    \caption{The illustration for {the} degrees of freedom in two dimensions, when $r = 1$, $k = 6$ and $r = 2$, $k = 9$.}
    \label{fig:example-2d}
\end{figure}

\begin{figure}[htbp]
    \includegraphics[]{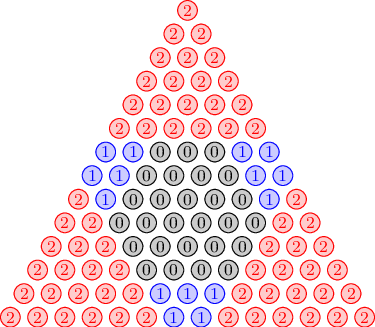}
    \qquad
    \includegraphics[]{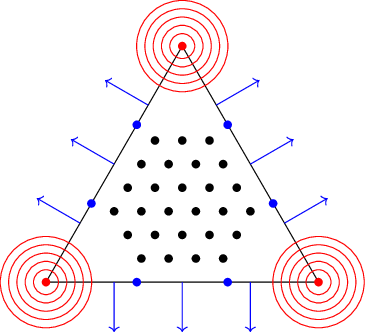}
    \caption{The decomposition of $\Sigma(\mathbb I_2, 13)$ with the continuity vector $\bm r = (1, 5)$ and polynomial degree $k = 13$, and corresponding degrees of freedom.}
    \label{fig:decomp-2d}
\end{figure}

\begin{remark}
\label{rmk:2d-completeness}
Two facts about the completeness are listed below without proof.
\begin{itemize}
        \item[-] The bubble spaces defined on a one-dimensional simplex (edge) are always complete. This holds for any dimension.
    \item[-] In two dimensions, the bubble function space defined on a two-dimensional simplex (face) is complete if and only if $r_2 + 1\le 2(r_1 + 1)$, which implies $\bm r = (r,2r)$ or $\bm r = (r, 2r + 1)$ for some nonnegative $r$.
\end{itemize}
\end{remark}

\subsection{\texorpdfstring{$C^r$}{Cr} finite element spaces in three dimensions}
\label{sec:3d-example}
This subsection provides the $C^r$ finite element spaces, defined in \eqref{eq:fe-dof} for three dimensions. For the case $r = 1$, it recovers the $\mathcal{P}_9-C^1$ \v Zen\' i\v sek element in \v Zen\'i\v sek \cite{vzenivsek1970interpolation}. The reader can also refer to \cite{zhang2009family, zhang2016family} for the $C^1$ and $C^2$ finite element methods in three dimensions, which extend the \v Zen\' i\v sek element, and can be derived from the construction in this paper as well. However, the construction of a family of $\mathcal{P}_{8r + 1}-C^r$ elements for $r \ge 1$, appearing in \cite[Chapter 18.11]{lai2007spline}, is different in nature from that given by this paper. In a very recent work \cite{zhang2022nodal}, Zhang provided a set of explicit basis functions for several $C^r$ finite element methods in three dimensions to verify the finite element methods constructed in this paper.

Given a continuity vector $\bm r = (r_1,r_2,r_3)$ and a polynomial degree $k$ satisfying \Cref{assu}, the degrees of freedom defined in \eqref{eq:fe-dof} for $u \in \Pcal_k(K)$ are as follows.

\begin{itemize}
    \item[-] The function value, first, second, $\cdots$, $r_3$-th order derivatives of $u$ at each vertex $\bm x$ of element $K$, the corresponding bubble function spaces are 
    $$\mathcal B_{\bm x,n,k} = \Span\{\lambda_{\bm x}^{k-n}\} \text{ for } 0 \le n \le r_3,$$
    where $\lambda_{\bm x}$ is the associated barycenter coordinate of vertex $\bm x$. This set of degrees of freedom is in fact defined by the set of multi-indices $\Sigma_3(\mathbb I_3,k)$.
    \item[-] The weighted moments $$\frac{1}{\vert e\vert }\int_{e}\left(\frac{\partial^{p+q}}{\partial \bm{n}_1^p \partial \bm{n}_2^q}u\right)\cdot v  \quad \forall v \in \mathcal{B}_{e,p+q,k} $$
          on each edge $e$ of element $K$, where $p \ge 0, q \ge 0$, $p+q \le r_2$, and 
         $$ \mathcal{B}_{e,p+q,k} = \Span\{\lambda_{e,2}^{\sigma_2} \lambda_{e,3}^{\sigma_3}\}= (\lambda_{e,2}\lambda_{e,3})^{r_3 - p - q}\mathcal{P}_{k-2(r_3+1)+p+q},$$
         with $\sigma_2 + \sigma_3 = k - p - q$ and $\sigma_2, \sigma_3 \ge r_3 - p - q$, where $\lambda_{e,2}$ and $\lambda_{e,3}$ are the two barycenter coordinates with respect to edge $e$. Here, $\bm{n}_1, \bm{n}_2$ are two linearly independent unit normal vectors of edge $e$ and they are perpendicular to each other. This set of degrees of freedom is defined by the set of multi-indices $\Sigma_2(\mathbb I_3,k)$.
    \item[-] The weighted moments
          \begin{equation*}
              \frac{1}{\vert F\vert }\int_{F} \left(\frac{\partial^n u}{\partial \bm n^n}\right) \cdot v~\quad \forall v \in \mathcal{B}_{F,n,k},
          \end{equation*}
          on each face $F$, $n = 0,1,\cdots, r_1$, where $\bm n$ is the unit outer normal vector of face $F$, and 
          $$\mathcal{B}_{F,n,k} = \Span\{\lambda_{F,1}^{\sigma_{1}}\lambda_{F,2}^{\sigma_2}\lambda_{F,3}^{\sigma_3} : (\sigma_1,\sigma_2,\sigma_3) \in \Sigma_0^{(r_2 - n, r_3 - n)}(\bm I,k - n)\},$$
          with $\bm I := \{1,2,3\}$, where $\lambda_{F,1},\lambda_{F,2}$ and $\lambda_{F,3}$ are the three barycenter coordinates with respect to face $F$. This set of degrees of freedom is defined by the set of multi-indices $\Sigma_1(\mathbb I_3,k)$.

    \item[-] The weighted moment(s) inside element $K$,
          \begin{equation*}
              \frac{1}{\vert K\vert} \int_{K} u \cdot v \quad \forall v \in \mathcal{B}_{K,0,k},
          \end{equation*}
          $$\mathcal{B}_{F,n,k} = \Span\{\lambda_{0}^{\sigma_{0}}\lambda_{1}^{\sigma_1}\lambda_{2}^{\sigma_2}\lambda_{3}^{\sigma_3} :  (\sigma_0,\sigma_1,\sigma_2,\sigma_3) \in \Sigma_0(\mathbb I_3,k)\},$$
          where $\lambda_0,\lambda_1, \lambda_2$ and $\lambda_3$ are the barycenter coordinates with respect to element $K$. This set of degrees of freedom is defined by the set of multi-indices $\Sigma_0(\mathbb I_3,k)$.

\end{itemize}

Consider a specific case $\bm{r} = (4, 8, 16)$ and $k = 33$. Given an element $K$, the shape function space is $\mathcal P_k(K)$. The number of degrees of freedom defined above are as follows:

\begin{enumerate}
    \item At each vertex, the number of degrees of freedom is $\binom{16+3}{3}$. Hence, the total number of degrees of freedom at the four vertices of element $K$ is $4\times\binom{19}{3} = 3876$,
          which is equal to $\card(\Sigma_3(\mathbb I_3,33))$.
    \item On each edge, the number of degrees of freedom is $\sum_{\theta=0}^8 \theta(\theta+1) = 240$. Hence, the total number of degrees of freedom on the six edges of element $K$ is $6\times 240 = 1440$, which is equal to the number of components of the set {$\Sigma_2(\mathbb I_3,33)$}, namely, $\card(\Sigma_2(\mathbb I_3,33))$.
    \item On each face, the number of degrees of freedom is $$\sum_{\theta=0}^4 \card(\Sigma_0^{(8-\theta,16-\theta)}(\mathbb I_2,33-\theta))= 28+45+63+82+102 = 320.$$ Hence,  the total number of degrees of freedom {on} the four faces of element $K$ is $4\times 320 = 1280$, which is equal to $\card(\Sigma_1(\mathbb I_3,33))$.
    \item Inside $K$, the set of degrees of freedom is corresponding to the set $\Sigma_0(\mathbb I_3, 33)$. The number of degrees of freedom inside element $K$ is 544.
\end{enumerate}


The rest of this section considers the bubble function spaces $\mathcal{B}_{F,n,k}$ from \Cref{thm:informal} inside the two dimensional faces $F$ of tetrahedron element $K$ for two lower order cases. Here let $\bm I = \{1,2,3\}$.

In the first case, $\bm r = (1,2,4)$ and $k= 9$. Then the corresponding bubble function spaces read
            $$\mathcal{B}_{F,0,9} = \Span\{\lambda_{F,1}^{\sigma_1}\lambda_{F,2}^{\sigma_2}\lambda_{F,3}^{\sigma_3} : (\sigma_1,\sigma_2,\sigma_3) \in \Sigma_0^{(2,4)}(\bm I,9) \} = (\lambda_{F,1}\lambda_{F,2}\lambda_{F,3})^3 \Pcal_{0}(F),$$
            and 
            $$\mathcal{B}_{F,1,9} =  \Span\{\lambda_{F,1}^{\sigma_1}\lambda_{F,2}^{\sigma_2}\lambda_{F,3}^{\sigma_3} : (\sigma_1,\sigma_2,\sigma_3) \in \Sigma_0^{(1,3)}(\bm I,8) \}  = (\lambda_{F,1}\lambda_{F,2}\lambda_{F,3})^2 \Pcal_{2}(F).$$

          In the second case, $\bm r = (2,4,8)$ and $k= 17$, the associated bubble function spaces as as follows,
            $$\mathcal{B}_{F,0,17} = \Span\{\lambda_{F,1}^{\sigma_1}\lambda_{F,2}^{\sigma_2}\lambda_{F,3}^{\sigma_3} : (\sigma_1,\sigma_2,\sigma_3) \in \Sigma_0^{(4,8)}(\bm I,17) \}= (\lambda_{F,1}\lambda_{F,2}\lambda_{F,3})^5 \Pcal_{2}(F),$$
            $$\mathcal{B}_{F,1,17} =  \Span\{\lambda_{F,1}^{\sigma_1}\lambda_{F,2}^{\sigma_2}\lambda_{F,3}^{\sigma_3} :(\sigma_1,\sigma_2,\sigma_3) \in \Sigma_0^{(3,7)}(\bm I,16) \}= (\lambda_{F,1}\lambda_{F,2}\lambda_{F,3})^4 \Pcal_{4}(F_l),$$
            and
            $$\mathcal{B}_{F,2,17}  =  \Span\{\lambda_{F,1}^{\sigma_1}\lambda_{F,2}^{\sigma_2}\lambda_{F,3}^{\sigma_3} : \Sigma_0^{(2,6)}(\bm I,15) \} = (\lambda_{F,1}\lambda_{F,2}\lambda_{F,3})^3 \Big(\Pcal_{6}(F) \setminus \Span\{\lambda_{F,0}^6,\lambda_{F,1}^6, \lambda_{F,2}^6\}\Big).$$
Note that the last bubble function space $\mathcal{B}_{F,2,17}$ cannot be regarded as a complete polynomial bubble function space.

\begin{remark}
The facts about the completeness of the bubble function spaces in three dimensions are listed below without proof:
\begin{itemize}
    \item[-] The bubble space defined on a two-dimensional simplex $F$ with respect to the $n-$th order normal derivative $(n \le r_1)$ is complete if and only if $r_3 = 2r_2 - n$ or $r_3 = 2r_2 -n + 1$, see \Cref{rmk:2d-completeness}. Therefore, the bubble function space can never be complete when $n \ge 2$.
    \item[-] In particular, when $n = 0$, the bubble function space on $F$ is complete if and only if $r_3 = 2r_2 $ or $r_3 = 2r_2 + 1$. When $n = 1$, the bubble function space on $F$ is complete if and only if $r_3 = 2r_2$. 
    \item[-] The bubble function space defined on a three-dimensional simplex (cell) is complete if and only if $\bm r = (0, 0, 0), (0, 0, 1), (0, 1, 2), (1, 2, 4), (1, 2, 5)$ or $(2, 4, 8)$. Hence, in almost all cases, the bubble function spaces are incomplete.
\end{itemize}
\end{remark}

\section{Proof of \Cref{thm:fe}: Unisolvency and Continuity}
\label{sec:proof}
This section proves \Cref{thm:fe}.
Recall the following refined decomposition of $\Sigma(\mathbb{I}_d, k)$
\begin{equation*}
    \Sigma(\mathbb I_d,k) = \bigcup_{\bm{N},n}\Sigma_{\bm{N},n}(\mathbb I_d,k),
\end{equation*}
with the set $\Sigma_{\bm{N}, n}(\mathbb{I}_d, k)$ defined in \Cref{defi:SigmaNn} above. Hence, a direct sum decomposition of the shape function space $\mathcal P_k(K)$ is as follows
\begin{equation}
    \label{eq:PdecomNn}
    \mathcal P_k(K) = \bigoplus_{\bm{N},n}\mathcal{P}_{\bm{N},n}.
\end{equation}
Here, 
$$\mathcal{P}_{\bm{N},n} := \Span\left\{\prod_{i\in \mathbb{I}_d}\lambda_i^{\alpha_i}:~ \alpha \in \Sigma_{\bm{N},n}(\mathbb I_d,k) \right\}.$$
This decomposition implies for any $u\in\mathcal{P}_k(K)$, there is a unique decomposition $u = \sum_{\bm{N},n} u_{\bm{N},n}$ with $u_{\bm{N},n} \in \mathcal{P}_{\bm{N},n}$, $\bm{N} \subsetneq \mathbb{I}_d$, $n \le r_{\card(\bm{N})}$.

The proof of unisolvency is based on an induction argument. To this end, an order of the pairs $(\bm{N}, n)$ will be defined below.
\begin{definition}[Order of the pair]
    \label{def:index-order}
    For all the pairs $(\bm{N}, n), \bm N \subseteq \mathbb{I}_d, n \in \mathbb{N}_0$, introduce the following order:
    Say $(\bm{N}', n') \preceq (\bm{N}, n)$ if $$\bm{N}' \supsetneq \bm{N}$$  or  $$\bm{N}' = \bm{N} \text{ and } n' \le n.$$
    Say $(\bm{N}', n') \prec (\bm{N}, n)$ if $(\bm{N}', n') \preceq (\bm{N}, n)$ and $(\bm{N}', n') \ne (\bm{N}, n)$.
\end{definition}

Again, consider the Argyris element as an example. For the three pairs $(\{1,2\},0)$, $(\{1,2\},1)$ and $(\{1\},0)$, it holds that $(\{1,2\},0) \prec (\{1,2\},1) \prec (\{1\},0)$.

This order leads to the following key lemma, which in particular tells that $\varphi_{\alpha}(\cdot)$ vanishes for $\Pcal_{\bm N,n}$ if the condition $(\bm N, n) \preceq (\bm N(\alpha), n(\alpha))$ does not hold.

\begin{lemma}[Induction Lemma]
    \label{lem:induction}
    Under \Cref{assu}, for $\alpha, \beta \in \Sigma(\mathbb{I}_d, k)$, let $\bm{N} := \bm{N}(\alpha), n := n(\alpha), \Delta := \mathbb{I}_d\setminus \bm{N}$. If the condition $(\bm{N}(\beta), n(\beta)) \preceq (\bm{N}, n)$ does not hold, then $\varphi_{\alpha}(\bm{\lambda}\beta) = 0$ with $\varphi_{\alpha}(\cdot)$ defined in \eqref{eq:dof:varphi-alpha}, where $\bm{\lambda}\beta := \prod_{i \in \mathbb{I}_d}\lambda_i^{\beta_i}$.
\end{lemma}

\begin{proof}
    To start the proof, note that \Cref{def:index-order} immediately implies that $\bm N$ cannot be empty. It follows from \Cref{defi-NDelta} that $n \le r_{\card(\bm N)}$. Then it can be asserted that $\sum_{i \in \bm N} \beta_i > n$. Otherwise, suppose that $\sum_{i \in \bm{N}}\beta_i \le n$, it holds that $\card(\bm{N}(\beta)) \ge \card(\bm{N})$. Since the condition $(\bm{N}(\beta), n(\beta)) \preceq (\bm{N}, n)$ does not hold, it implies that $\bm N$ is not a subset of $\bm N(\beta)$. Therefore, it follows that $$\card(\bm{N}(\beta) \cup \bm{N}) \ge \card(\bm{N}(\beta)) + 1$$ and
    \begin{equation*}
        \begin{split}
            \sum_{i \in \bm{N}(\beta) \cup \bm{N}} \beta_i &\le \sum_{i \in \bm{N}(\beta)} \beta_i + \sum_{i \in \bm{N}} \beta_i \\ & \le r_{\card(\bm{N}(\beta))} + r_{\card(\bm{N})} \\ & \le 2r_{\card(\bm{N}(\beta))} \le r_{\card(\bm{N}(\beta) \cup \bm{N})},
        \end{split}
    \end{equation*}
    which contradicts with the definition of $\bm{N}(\beta)$, since $\bm{N}(\beta) \cup \bm N$ satisfies the condition in \Cref{defn:decomp}, while $\bm{N}(\beta)$ is the largest {admissible} choice. Here the second inequality is from $\card(\bm{N}(\beta)) \ge \card(\bm{N})$, while the last inequality is from $\card(\bm{N}(\beta) \cup \bm{N}) \ge \card(\bm{N}(\beta)) + 1$.

    Let $F := \lrangle{\Delta} := \mathrm{conv}\{\bm{x}_i :~ i\in \Delta\}$ with the unit normal vectors $\bm n_{F,i}, i \in \bm N$.
    Note that each barycenter coordinate $\lambda_i$ associated with vertex $\bm x_i$ vanishes on $F$ for $i \in \bm{N}$.
    With
    $$\bm \lambda^{\Delta} \beta := \prod_{i \in \Delta} \lambda_i^{\beta_i}, \quad \bm \lambda^{\bm N} \beta := \prod_{i \in \bm N} \lambda_i^{\beta_i},$$
    a direct calculation yields
    \begin{equation}
        \label{eq:single-cal}
        \frac{\partial^{n'}}{ \prod_{i \in \bm{N}}\partial\bm n_{F, i}^{\theta'_i}} \bm{\lambda}^{\bm{N}} \beta \Big|_F= \frac{\partial^{n'}}{ \prod_{i \in \bm{N}}\partial\bm n_{F, i}^{\theta'_i}} \left( \prod_{i \in \bm N}\lambda_i^{\beta_i} \right) \Big|_F = 0 \text{ for } n' \le n \text{ and } \theta' \in \Sigma(\bm{N}, n').
    \end{equation}
    It follows from \eqref{eq:single-cal} and the generalized Leibniz rule that
    \begin{equation*}
        \begin{split} \frac{\partial^{n}}{ \prod_{i \in \bm{N}}\partial\bm n_{F, i}^{\alpha_i}} \bm{\lambda} \beta \Big|_F  = \sum_{n' \le n, \theta' \in \Theta_{\bm{N}, n', \theta}} \binom{\theta}{\theta'} \left( \frac{\partial^{n'}}{ \prod_{i \in \bm{N}}\partial\bm n_{F, i}^{\theta'_i}} \bm{\lambda}^{\bm{N}} \beta \right) \left( \frac{\partial^{n - n'}}{\prod_{i \in \bm{N}}\partial\bm n_{F, i}^{\theta_i - \theta'_i}} \bm{\lambda}^{\Delta} \beta \right) \Big|_F= 0.
        \end{split}
    \end{equation*}
    Here $\theta$ is the first component of the pair $(\theta, \sigma) = \mathcal{R}_{\bm{N}, \Delta}(\alpha)$ defined in \Cref{prop:bijection}, and the set $\Theta_{\bm{N}, n', \theta} := \{\theta' \in \Sigma(\bm{N}, n') :~ \theta'_i \le \theta_i, i \in \bm{N}\}$. This implies that $\varphi_{\alpha}(\bm{\lambda}\beta) = 0$.
\end{proof}

The following lemma is also crucial in the proof of unisolvency, indicating that for the subspace $\mathcal{P}_{\bm N, n}$ of the shape function space, the degrees of freedom $\{\varphi_{\alpha}(\cdot) : \alpha \in \Sigma_{\bm N, n}(\mathbb I_d,k)\}$ are unisolvent.

\begin{lemma}[Unisolvency for the subspace $\Pcal_{\bm N,n}$]
    \label{lem:uniNn}
    Under \Cref{assu}, for any nonempty subset $\Delta \subseteq \mathbb{I}_d$ and $F := \lrangle{\Delta} := \mathrm{conv}\{\bm{x}_i:~ i\in \Delta\}$, let $\bm{N} := \mathbb{I}_d \setminus \Delta$, $n\le r_{\card(\bm{N})}$ and $u_{\bm{N},n} \in \Pcal_{\bm N,n}$. If  $\varphi_{\alpha}(u_{\bm{N},n}) = 0$ for all $\alpha \in \Sigma_{\bm N, n}(\mathbb I_d,k)$, then $u_{\bm{N},n} = 0$.
\end{lemma}

\begin{proof}
    First, consider the case $\bm{N} \ne \emptyset$. A new basis of the space $\Pcal_{\bm N,n}$ corresponding to the set $\{\bm{n}_{F, i} : i \in \bm N\}$ will be introduced.
    This new basis of the space $\Pcal_{\bm N,n}$ depends on a new basis $\{\nu_i:~ i \in \bm N\}$ of the space $\Span\{\lambda_i :~ i \in \bm N\}$, such that
    $$\grad \nu_i = \bm{n}_{F, i},\quad \forall i \in \bm{N}.$$
    In fact, the new basis $\{\nu_i:~ i \in \bm N\}$ can be constructed as follows.

    Suppose $\bm{m}_I$ is the unique (outer) normal vector of codimension 1 subsimplex $\lrangle{\mathbb{I}_d\setminus \{I\}} := \mathrm{conv}\{\bm{x}_i:~ i\in \mathbb{I}_d\setminus\{I\}\}$ for $I\in \mathbb{I}_d$. It follows from the definition that the set $\{\bm{n}_{F, i} : i \in \bm N\}$ and the set $\{\bm{m}_{i} :~ i \in \bm N\}$ are two bases of the space perpendicular to subsimplex $F$. Then there exist $c_{ij} \in \mathbb R$, $i, j\in\bm{N}$, such that
    \begin{equation*}
        \bm{n}_{F, i} = \sum_{j\in\bm{N}}c^{ij}\bm{m}_j,\quad i\in\bm{N}.
    \end{equation*}
    Clearly, there exists $\xi_i \ne 0$ such that $$\grad \lambda_i = \xi_i\bm{m}_i,\quad  i \in \bm{N}$$ from the definition of $\bm{m}_i$. Define
    \begin{equation*}
        \nu_i := \sum_{j \in \bm{N}} \frac{c^{ij}}{\xi_j}\lambda_j, \quad i \in \bm{N},
    \end{equation*}
    then it holds that
    $$\grad \nu_i = \bm{n}_{F, i},\quad  i \in \bm{N}.$$
    Since these vectors $\bm{n}_{F, i}$, $i\in \bm{N}$, are linearly independent, the functions $\{\nu_i:~ i \in \bm N\}$ form a basis of the space $\Span\{\lambda_i :~ i \in \bm N\}$. Thus, a new basis of $\mathcal{P}_{\bm{N}, n}$ can be defined as follows.

    For $\alpha \in \Sigma_{\bm{N}, n}(\mathbb{I}_d, k)$, define $\bm{\lambda}_{\textrm{nor}}: \Sigma_{\bm{N}, n}(\mathbb{I}_d, k) \to \mathcal{P}_{\bm{N}, n}$ such that
    $$\bm{\lambda}_{\textrm{nor}} \alpha := \bm{\lambda}^{\Delta}\alpha \prod_{i \in \bm{N}} \nu_i^{\alpha_i}.$$
    It follows from \Cref{prop:bijection} that $\{\bm{\lambda}_{\textrm{nor}} \alpha :~ \alpha \in \Sigma_{\bm{N}, n}(\mathbb{I}_d, k)\}$ is a basis of $\mathcal{P}_{\bm{N}, n}$, where
    $\bm{\lambda}^{\Delta}\alpha := \prod_{i \in \Delta}\lambda_i^{\alpha_i}.$
    Hence $u_{\bm{N}, n}$ can be reexpressed as $$u_{\bm{N}, n} := \sum_{\alpha \in \Sigma_{\bm{N}, n}(\mathbb{I}_d, k)} c_{\alpha} \bm{\lambda}_{\textrm{nor}}\alpha$$ for combination parameters $c_{\alpha}$, $\alpha \in \Sigma_{\bm N,n}(\mathbb I_d,k)$. Next, define $$v_{\bm{N}, n, \theta} := \sum_{\alpha \in \Sigma_{\bm{N}, n, \theta}(\mathbb{I}_d, k)} c_{\alpha} \bm{\lambda}^{\Delta}\alpha$$ for $\theta \in \Sigma(\bm{N}, n)$. It follows from \Cref{prop:bijection} that the mapping $\mathcal{R}_{\bm{N}, \Delta}$ is a bijection between $\Sigma_{\bm N, n}(\mathbb I_d,k)$ and $\Sigma(\bm N, n) \times \Sigma_0^{(\bm q)}(\Delta, k - n)$. Hence, it holds that
    \begin{equation*}
        \begin{split}
            u_{\bm{N}, n} &=  \sum_{\alpha \in \Sigma_{\bm{N}, n}(\mathbb{I}_d, k)} c_{\alpha} \bm{\lambda}_{\textrm{nor}}\alpha \\ &= \sum_{\theta \in \Sigma(\bm{N}, n)} \sum_{\alpha \in \Sigma_{\bm{N}, n, \theta}(\mathbb{I}_d, k)} (c_{\alpha} \bm{\lambda}^{\Delta}\alpha \prod_{i \in \bm{N}} \nu_i^{\theta_i})\\ & = \sum_{\theta \in \Sigma(\bm{N}, n)} (v_{\bm{N}, n, \theta} \prod_{i \in \bm{N}} \nu_i^{\theta_i}).
        \end{split}
    \end{equation*}

    Since it holds that
    $$\frac{\partial \nu_i}{\partial \bm{n}_{F, j}} = \delta_{ij}, \quad \forall i,j \in \bm{N},$$ with $\delta_{ij}$ being Kronecker's delta, and that
    $$\frac{\partial \nu_i}{\partial \bm m} = 0,\quad \forall i \in \bm{N}$$ for any vector $\bm{m}$ such that $\bm m \perp \bm n_{F, j}$ for all $j \in \bm N$, it follows that given $\theta \in \Sigma(\bm{N}, n)$, for $\alpha \in \Sigma_{\bm{N}, n, \theta}(\mathbb I_d,k)$ and $\beta \in \Sigma_{\bm{N}, n}(\mathbb I_d,k)$,
    \begin{equation}
        \label{eq:identity-3} \frac{\partial^n}{\prod_{i\in \bm{N}} \partial \bm{n}_{F, i}^{\alpha_i}} \bm{\lambda}_{\textrm{nor}} \beta \big\vert _{F} = \left\{\begin{aligned} 0 , &\quad \beta \notin \Sigma_{\bm{N},n,\theta}(\mathbb I_d,k), \\ \theta! \bm{\lambda}^{\Delta} \beta , &\quad \beta \in \Sigma_{\bm{N},n,\theta}(\mathbb I_d,k), \end{aligned} \right.
    \end{equation}
    where $\Sigma_{\bm N,n,\theta}(\mathbb I_d,k)$ is defined in \eqref{eq:Sigma-Nn-theta} above.

    Now given $\theta \in \Sigma(\bm{N},n)$, for $\alpha \in \Sigma_{\bm{N},n,\theta}(\mathbb I_d,k)$ and $\beta \in \Sigma_{\bm{N},n}(\mathbb I_d,k)$, it follows from \eqref{eq:identity-3} that
    \begin{equation*}
        \varphi_{\alpha}(\bm{\lambda}_{\textrm{nor}} \beta) =
        \left\{\begin{aligned} 0, & \quad \beta \notin \Sigma_{\bm{N},n,\theta}(\mathbb I_d,k), \\ \frac{\theta!}{\left\vert F\right\vert }\langle \bm{\lambda}^{\Delta}\beta, \bm{\lambda}^{\Delta}\alpha \rangle_{F}, & \quad \beta \in \Sigma_{\bm{N},n,\theta}(\mathbb I_d,k).\end{aligned}\right.
    \end{equation*}
    Then the linearity of $\varphi_{\alpha}(\cdot)$ {gives}
    \begin{equation*}
        \begin{split}
            \varphi_{\alpha}(u_{\bm{N},n}) &=   \sum_{\beta \in \Sigma_{\bm{N}, n, \theta}(\mathbb{I}_d, k)} c_{\beta} \frac{\theta!}{\left\vert F\right\vert }\langle \bm{\lambda}^{\Delta}\beta, \bm{\lambda}^{\Delta}\alpha \rangle_{F} \\ & = \frac{\theta!}{\left\vert F\right\vert }\langle v_{\bm{N}, n, \theta}, \bm{\lambda}^{\Delta}\alpha \rangle_{F}.
        \end{split}
    \end{equation*}
    Notice that $\varphi_{\alpha}(u_{\bm{N},n}) = 0$ for $\alpha \in \Sigma_{\bm{N}, n}(\mathbb{I}_d, k)$. This leads to
    \begin{equation}
        0 = \sum_{\alpha \in \Sigma_{\bm{N}, n, \theta}(\mathbb{I}_d, k)}c_{\alpha} \varphi_{\alpha}(u_{\bm{N}, n}) = \frac{\theta!}{\left\vert F\right\vert }\langle v_{\bm{N}, n, \theta}, v_{\bm{N}, n, \theta} \rangle_{F}.
    \end{equation}
    It yields $v_{\bm{N}, n, \theta} = 0$. Since $u_{\bm{N}, n} = \sum_{\theta \in \Sigma(\bm{N}, n)} (v_{\bm{N}, n, \theta} \prod_{i \in \bm{N}} \nu_i^{\theta_i})$, it concludes that $u_{\bm{N}, n} = 0$.

    Second, consider the case $\bm{N} = \emptyset$ and $n = 0$. Then it holds that $\Sigma_{\bm{N}, n}(\mathbb{I}_d, k) = \Sigma_0(\mathbb{I}_d, k)$. Thus, $u_{\bm{N}, n}$ can be rewritten as $u_{\bm{N}, n} = \sum_{\alpha \in \Sigma_0(\mathbb{I}_d, k)}c_{\alpha}\bm{\lambda}\alpha$ for combination parameters $c_{\alpha}, \alpha \in \Sigma_0(\mathbb I_d,k)$. Here
    $$\bm{\lambda}\alpha := \prod_{i \in \mathbb{I}_d} \lambda_i^{\alpha_i}.$$
    Notice that $\varphi_{\alpha}(u_{\bm{N},n}) = 0$ holds for $\alpha \in \Sigma_{\bm{N}, n}(\mathbb{I}_d, k)$. {This yields}
    \begin{equation}
        \begin{split}
            0  & = \sum_{\alpha \in \Sigma_0(\mathbb{I}_d, k)}c_{\alpha} \varphi_{\alpha}(u_{\bm{N}, n})\\ & = \sum_{\alpha \in \Sigma_0(\mathbb{I}_d, k)}c_{\alpha} \frac{\theta!}{\left\vert F\right\vert }\langle u_{\bm{N}, n}, \bm{\lambda} \alpha \rangle_{F} \\ & = \frac{\theta!}{\left\vert F\right\vert }\langle u_{\bm{N}, n}, u_{\bm{N}, n} \rangle_{F}.
        \end{split}
    \end{equation}
    Thus $u_{\bm{N}, n} = 0$.
\end{proof}

Now it is ready to show the unisolvency.

\begin{proposition}[Unisolvency]
    \label{prop:unisolvent}
    Given an element $K \in \mathcal{T}$, the set of degrees of freedom $\{\varphi_{\alpha}(\cdot):\alpha \in \Sigma(\mathbb I_d,k)\}$ defined in \eqref{eq:fe-dof-local}, is unisolvent for the shape function space $\mathcal{P}_k(K)$ if \Cref{assu} is satisfied.
\end{proposition}

\begin{proof}
    From the definition of the intrinsic decomposition of the set $\Sigma(\mathbb I_d,k)$, the dimension of the shape function space and the number of degrees of freedom coincide. Hence it suffices to show if $u \in \mathcal P_k(K)$ vanishes on all degrees of freedom $\varphi_{\alpha}(\cdot)$ for $\alpha \in \Sigma(\mathbb I_d,k)$, then $u = 0$.  Suppose that $\varphi_{\alpha}(u)=0$ for all $\alpha\in\Sigma(\mathbb{I}_d, k)$ for some $u\in\mathcal{P}_{k}(K)$. Recall the decomposition \eqref{eq:PdecomNn} of $\mathcal P_k(K)$, namely $u = \sum_{\bm{N},n}u_{\bm{N},n}$ with $u_{\bm{N},n}\in \mathcal{P}_{\bm{N},n}$.

    The proof is based on performing a mathematical induction. To begin with, consider the minimal element $(\bm N, n)$ with respect to index order $\prec$, i.e., for all $(\bm N', n')$, the condition $(\bm N', n') \preceq (\bm N, n)$ does not hold. By \Cref{lem:induction}, for any $\alpha \in \Sigma_{\bm N,n}(\mathbb I_d,k)$, it follows that
    \begin{equation*}
        0 = \varphi_{\alpha}(u) = \sum_{(\bm N', n') \not\preceq (\bm N,n)}\varphi_{\alpha}(u_{\bm N',n'}) + \varphi_{\alpha}(u_{\bm N, n}) = \varphi_{\alpha}(u_{\bm N, n}).
    \end{equation*}
    Therefore by \Cref{lem:uniNn}, it can be concluded that $u_{\bm N,n} = 0$.

    Suppose that for all $(\bm N', n') \prec (\bm N,n)$, it holds that $u_{\bm N',n'} = 0.$ Then, again, by \Cref{lem:induction}, it follows that
    \begin{equation}
        0 = \varphi_{\alpha}(u) =  \sum_{(\bm N', n') \not\preceq (\bm N,n)}\varphi_{\alpha}(u_{\bm N',n'})  + \sum_{(\bm N', n') \prec (\bm N,n)}\varphi_{\alpha}(u_{\bm N',n'})+ \varphi_{\alpha}(u_{\bm N, n}) = \varphi_{\alpha}(u_{\bm N, n}).
    \end{equation}
    Therefore by \Cref{lem:uniNn}, it can be concluded that $u_{\bm N,n} = 0$.

    It then follows from the mathematical induction on the partial order sets that all $u_{\bm N,n} = 0$. Hence $u = 0$, which implies the unisolvency.
\end{proof}

\begin{proposition}[Continuity]
    \label{prop:continuity}
    Let $F$ be a $(d-1)$-dimensional simplex shared by two $d$-dimensional simplices $K^+$ and $K^-$. Let functions $u^{\pm}$ be defined on $K^{\pm}$, respectively. Suppose that they are compatible on their common degrees of freedom. Then the piecewise polynomial $u$ defined as $u = u^{+}$ on $K^+$ and $u = u^-$ on $K^-$ is of $C^{r_1}(K^+\cup K^-)$.
\end{proposition}

\begin{proof}
    It suffices to prove that: if all degrees of freedom $\varphi_{\alpha}(\cdot)$ associated to $F$ (including the degrees of freedom defined inside itself and/or any subsimplex of it, namely $\lrangle{\Delta(\alpha)}$ is a subsimplex of $F$) vanish, then it holds that
    \begin{equation}
        u = \frac{\partial u}{\partial \bm{n}}  = \frac{\partial^2 u}{\partial \bm{n}^2} = \cdots \frac{\partial^{r_1} u}{\partial \bm{n}^{r_1}} = 0 \text{ on } F.
    \end{equation}
    Here $\bm{n}$ is the unit normal vector of $F$. Without loss of generality, assume that $F = \lrangle{\mathbb I_d \setminus \{d\}} = \langle \mathbb I _{d-1}\rangle$. By definition, $\lrangle{\Delta(\alpha)} \subseteq F$ implies that $d \in \bm N(\alpha)$.

    For any nonempty subset $\Delta \subseteq \mathbb{I}_{d-1}$ and $f := \lrangle{\Delta}$, let $\bm{N} := \mathbb{I}_d\setminus \Delta$, $s : = \card(\bm N)$. Consider $\varphi_{\theta, \sigma}$ from \eqref{eq:theta-sigma}, there holds that $\varphi_{\theta, \sigma}(u) = 0$ for $\theta \in \Sigma(\bm N,n)$ and $\sigma \in \Sigma_0^{(\bm q)}(\Delta, k-n)$ defined in \eqref{eq:sigma0q}, where $\bm q = (r_{s+1} - n,\cdots, r_{d} - n)$. Taking $\bm N' = \bm N \setminus \{ d\}, 0 \le l \le r_1$ and $v = \frac{\partial^l}{\partial \bm n^{l}} u$, it then follows that
    \begin{equation}
        \frac{1}{\left\vert f\right\vert } \bigg \langle (\frac{\partial^{n - l}}{ \prod_{i\in \bm{N'}}\partial \bm{n}_{f, i}^{\tilde \theta_i}} v)\big\vert_{f}, \bm{\lambda}^{\Delta}\sigma \bigg \rangle_{f}
    \end{equation}
    vanishes for $\tilde \theta \in \Sigma(\bm N', n - l)$ and $\sigma \in \Sigma_0^{(\bm q)}(\Delta, k-n)$,
    where $\bm \lambda^{\Delta}\sigma := \prod_{i \in \Delta} \lambda_i^{\sigma_i}$.
    It is straightforward to see that the continuity vector $\bm{p}_l = (r_2 - l, r_3 - l,\cdots r_d - l)$ and the polynomial degree $k - l$ satisfies \Cref{assu}. By the unisolvency with respect to the continuity vector $\bm{p}_l$ and the polynomial degree $k - l$, it holds that $v = 0$, implying the continuity.

\end{proof}

\section{Generalizations}
\subsection{Discontinuous Elements}
\label{sec:l2space}
This section discusses a simple case where \Cref{assu} is violated. This yields a construction of discontinuous elements. From now on, the following assumption is considered to replace \Cref{assu}.

\begin{assumption} 
\label{assu2}
For the continuity vector $\bm{r} = (r_1, \cdots, r_d)$ {such that} $r_1, \cdots, r_{t-1} = -1$ {for $1 < t < d$} and the polynomial degree $k$, it holds that
    \begin{equation*}
        r_d \ge 2r_{d-1} \ge 4r_{d-2} \ge \cdots \ge 2^{d-t}r_{t}
    \end{equation*}
    and
    \begin{equation*}k \ge 2r_d+1.\end{equation*}
\end{assumption}

Under \Cref{assu2}, the {corresponding} intrinsic decomposition of $\Sigma(\mathbb{I}_d, k)$ and $\Sigma(\bm{I}, k)$ can be defined in a similar way as \Cref{defn:decomp} and \Cref{rmk:decomp}. In particular, $\Sigma_s(\mathbb{I}_d, k) = \emptyset$ for $1 \le s < t$.

For $\alpha \in \Sigma_s(\mathbb{I}_d, k)$ with $s = 0$ and $s \ge t$, the definitions of $\bm{N}(\alpha)$, $\Delta(\alpha)$ and $\Delta(\alpha)$ are the same as those in \Cref{defi-NDelta}. The uniqueness of $\bm{N}(\alpha)$ and $\Delta(\alpha)$ can be proved by a similar 
{argument} in \Cref{prop:uni-N-Delta}.

\subsection{Stokes Complex in two dimensions}
\label{sec:Stokes}
As an application, consider the following smoothing de Rham complex in two dimensions. Given a conforming triangular grid $\mathcal T = \mathcal T(\Omega)$ of the two-dimensional polygonal domain $\Omega$, denote the global finite element space defined in \Cref{thm:fe} with the continuity vector $\bm{r}$ as $V_k^{(\bm{r})}(\mathcal T)$. Here and in the next subsection, the superscript $\bm r$ will be used to emphasize the dependency of these finite element spaces on the continuity vector.

Note that the functions in $V_k^{(\bm{r})}$ are of $C^{r_2}$ continuity across the vertices and of $C^{r_1}$ continuity for two dimensions.
\begin{proposition}
    \label{prop:derham}
    Let $\bm{r} = (r_1,r_2)$, $\bm{r'} = (r_1 - 1, r_2 - 1)$, $\bm{r''} = (r_1 - 2, r_2 - 2)$ with $r_1 \ge 1$. Suppose $r_2 \ge 2r_1$ and $k \ge 2r_2+1$. Then, it holds that the following sequence
    \begin{equation}
        \label{eq:smooth-de-rham}
        \mathbb R \stackrel{\hookrightarrow}{\longrightarrow}V_{k}^{(\bm{r})}(\mathcal T) \stackrel{\curl}{\longrightarrow} \big[V_{k-1}^{(\bm{r'})}(\mathcal T)\big]^2 \stackrel{\div}{\longrightarrow} V_{k-2}^{(\bm{r''})}(\mathcal T)\stackrel{}{\longrightarrow} 0
    \end{equation}
    is a complex and exact, provided that the domain $\Omega$ is simply connected.
\end{proposition}


\begin{remark}
In \Cref{prop:derham}, the continuity vector $\bm{r''}$ might be $(-1,r_2'')$ for some nonnegative integer $r_2''$, which is a special case in \Cref{sec:l2space}. In particular, the case with $r_1 = 1, r_2 = 2$ in the above sequence \eqref{eq:smooth-de-rham} recovers the complex constructed in Falk and Neilan \cite{falk2013stokes}.
\end{remark}

\begin{proof}
    It is straightforward to see that \eqref{eq:smooth-de-rham} is a complex. To show the exactness of \eqref{eq:smooth-de-rham}, it suffices to prove the discrete kernel of $\div$ is just the discrete image of $\curl$, and to compute the dimensions of these finite element spaces in \eqref{eq:smooth-de-rham}. 
    
    Notice that the dimension of $\mathcal P_{k}^{(\bm{r})}(\mathcal T)$ is just the total number of degrees of freedom defined at the vertices, on the edges and in the interior of element $K$. Denote the number of vertices, edges and faces by $V,E,F$, respectively. At each vertex, the number of degrees of freedom is $a_V = \binom{r_2+2}{2}$. On each edge, the number of degrees of freedom is
    \begin{equation*}
        \begin{split}a_E &= \sum_{m = 0}^{r_1}(k + m - 2r_2 - 1) = \frac{1}{2} (2k - 4r_2 - 2 + r_1)(r_1+1)\\ & = (k - 2r_2 - 1)(r_1+1) + \binom{r_1+1}{2}.
        \end{split}
    \end{equation*}
    Inside each element, the number of degrees of freedom is $\binom{k+2}{2} - a_V - a_E$. As a summary, this gives
    \begin{equation*}
        \dim(V_{k}^{(\bm{r})}(\mathcal T)) = \binom{r_2 + 2}{2} V + \sum_{m=0}^{r_1}(k+m-2r_2-1) E + \left(\binom{k+2}{2} - a_V - a_E\right)F.
    \end{equation*}
    Since $$\binom{n}{2} - 2\binom{n-1}{2} + \binom{n-2}{2} = 1$$ and
    $$(k-2r_2-1)(r_1 + 1) - 2(k-2r_2)(r_1) + (k-2r_2+1)(r_1 - 1) = -2,$$
    by Euler's formula it follows that
    \begin{equation*}
        \begin{split}
            \dim(V_{k}^{(\bm{r})}(\mathcal T)) - &
            2\dim(V_{k-1}^{(\bm{r}')}(\mathcal T)) + \dim(V_{k-2}^{(\bm{r}'')}(\mathcal T)) \\
            =
            & V -E + F = 1.
        \end{split}
    \end{equation*}

It remains to show that the discrete kernel of $\div$ is just the discrete image of $\curl$. Suppose that for some $\bm v \in [V_{k-1}^{(\bm{r}')}]^2$ such that $\div \bm v = 0$, then by the exactness of the continuous Stokes complex, there exists $\phi \in H^2(\Omega)$ such that $\curl \phi = \bm v$. By Sobolev's embedding, $\phi$ is continuous. Restricting this identity to each element $K$ immediately shows that $\phi$ is a polynomial of degree at most $k$ in each element $K$. It remains to show that $\phi$ satisfies the required continuity. Since $\bm v = \curl \phi$ is of $C^{r_2-1}$ continuity, it follows that $\phi$ is of $C^{r_2}$ continuity at each vertex. Similarly, it can be found that $\bm v$ is of $C^{r_1}$ continuity across each edge.

Therefore, it holds that $\ker \div\subseteq \im \curl$ on the discrete kernel. The converse inclusion is from the definition of the complex. Hence it must hold that $\ker \div = \im \curl$ on the discrete level. Therefore, the complex \eqref{eq:smooth-de-rham} is exact.
\end{proof}

\subsection{\texorpdfstring{$H(\div)$}{H(div)} Finite Element Space: A Generalized Stenberg Element}
\label{sec:Hdiv}

In this section, it is assumed that the continuity vector $\bm r = (r_1, r_2)$ satisfies the conditions: $r_1 = l + 1/2$ for some integer $l \ge -1$, and $r_2 \ge 2r_1 + 1$ is a nonnegative integer. It is also assumed that the polynomial degree $k \ge 2r_2+1$. For edge $e$, denote by $\bm n$ the (outer) normal vector and $\bm t$ the tangential vector. 

This section is motivated by the triangular Stenberg element \cite{stenberg1990error}, where $\bm r$ is chosen as $(-\frac{1}{2}, 0)$. The shape function space for the Stenberg element is $[\mathcal P_2(K)]^2$. Given $u \in [\mathcal{P}_2(K)]^2$, the degrees of freedom are as follows:
\begin{itemize}
\item[-] The value of $\bm u(\bm x)$ at each vertex $\bm x$.
\item[-] $\int_e(\bm u \cdot\bm n)$ on each edge $e$.
\item[-] $\int_K \bm u \cdot \bm q $ for $\bm q \in \mathcal{RT}_0$, where $\mathcal{RT}_0$ is the lowest order Raviart--Thomas element. 
\end{itemize}

For the global Stenberg element space, there is another characterization, see \cite{christiansen2018nodal}. Explicitly, the Stenberg element space $\bm{\mathcal{S}}$ admits the following decomposition
\begin{equation}
\bm {\mathcal S}(\mathcal T) := [\mathcal P_2(\mathcal T)]^2 \oplus \bm{\mathcal B}_{\div}(\mathcal T).
\end{equation}
Here $\mathcal P_2(\mathcal T)$ is the standard $H^1$ conforming quadratic Lagrange element space, while $\bm{\mathcal{B}}_{\div}(\mathcal T)$ is the elementwise $H(\div)$ bubble function space characterized by 
\begin{equation}
\bm{\mathcal{B}}_{\div}(K) := \{ \bm u \in [\Pcal_2(K)]^2 :~\bm u\cdot \bm n = 0 \text{ on each edge } e\}. 
\end{equation}
and the global $H(\div)$ bubble function space is defined as
\begin{equation}
\bm{\mathcal{B}}_{\div}(\mathcal T) := \{ \bm u \in L^2(\Omega) : u|_K \in \bm{\mathcal B}_{\div}(K), \forall K \in \mathcal T\}.
\end{equation}

It can be proved that, for each element $K$, the dimension of $\bm{\mathcal B}_{\div}(K)$ is 3, and the last degrees of freedom can be modified as $\int_K \bm u \cdot \bm q$ for $\bm q \in \bm{\mathcal B}_{\div}(K)$. For further information, the interested readers can refer to \cite{christiansen2018nodal} for the higher degree case and \cite{hu2015family,hu2021conforming} for some tensor generalization.

Given the continuity vector $\bm r = (r_1,r_2)$, a generalized $H(\div)$ conforming finite element space $\bm U$ is constructed in this section such that the following requirements are fulfilled: For $\bm u \in \bm U$, $\bm u$ {is of} $C^{r_2}$ continuity across the vertex, and $\bm u \cdot \bm n$ {is of} $C^{r_1 + 1/2}$ continuity across the internal edges, while $\bm u \cdot \bm t$ {is of} $C^{r_1 - 1/2}$ continuity across the internal edges. To this end, it is natural to introduce the following generalized $H(\div)$ bubble function space.

A new generalized bubble function space $\bm{\mathcal{B}}_{\div,k}^{(\bm r)}(K)$ is defined as 
\begin{equation}
    \begin{split}
    \bm{\mathcal{B}}_{\div,k}^{(\bm r)}(K):= \{ \bm u \in [\mathcal P_k(K)]^2 ~|~ & D^{\alpha}\bm u(\bm x) = 0 \text{ for } |\alpha| \le r_2 \text{ at each vertex } \bm x \in K,\\ &(D^{\beta}) \bm u\cdot \bm n|_e = 0 \text{ for } |\beta| \le r_1+ \frac{1}{2} \text{ on each edge }e\text{ of } K, \\ & (D^{\gamma})(\bm u \cdot \bm t)|_e = 0 \text{ for } |\gamma| \le r_1 - \frac{1}{2} \text{ on each edge }e\text{ of } K\}.
    \end{split}
\end{equation}

Then the degrees of freedom are defined as follows. 
\begin{definition}
    \label{defn:generlized-stenberg}
    Given $\bm{u} \in [\Pcal_k(K)]^2$, the degrees og freedom are as follows 
\begin{itemize}
\item[-] $D^{\alpha} \bm u(\bm x)$ at vertex $\bm x$ {of $K$}, for $|\alpha| \le r_2$.
\item[-] $\int_{e} \frac{\partial}{\partial \bm n^m} (\bm u \cdot \bm n) q$ for $q \in \Pcal_{k -2r_2 - m}$, $m = 0,1,\cdots, r_1+\frac 12$ {on each edge $e$ of $K$.}
\item[-] $\int_{e} \frac{\partial}{\partial \bm n^m} (\bm u \cdot \bm t) q$ for $q \in \Pcal_{k - 2r_2 - m}$, $m = 0,1,\cdots, r_1 - \frac 12$ {on each edge $e$ of $K$.}
\item[-] $\int_K \bm u\cdot \bm q$ for $\bm q \in \bm{\mathcal{B}}_{\div,k}^{(\bm r)}(K)$.
\end{itemize}
\end{definition}
It is not easy to write down the explicit form {of functions} in $\bm{\mathcal{B}}_{\div,k}^{(\bm r)}(K)$. Nevertheless, it is possible to count the dimension of $\bm{\mathcal{B}}_{\div,k}^{(\bm r)}(K)$. Since $k_2 \ge 2(k_1+\frac 12)$, the degrees of freedom of the first three sets of degrees of freedom in \Cref{defn:generlized-stenberg} are linearly independent. Therefore, the degrees of freedom defined in \Cref{defn:generlized-stenberg} are unisolvent for the shape function space $[\mathcal P_k(K)]^2$. Moreover, the resulting global finite element space admits the continuity in the following proposition.

\begin{proposition}
    \label{prop:derham2}
    Let $\bm{r} = (r_1+1/2,r_2+1)$, $\bm{r'} = (r_1, r_2)$, $\bm{r''} = (r_1 -1/2, r_2 - 1)$ with $r_1 \ge - 1/2$. Suppose $r_2 \ge 2r_1+1$ and $k \ge 2r_2+1$. Then, it holds that the following sequence
    \begin{equation}
        \label{eq:smooth-de-rham2}
        \mathbb R \stackrel{\hookrightarrow}{\longrightarrow}V_{k+1}^{(\bm{r})}(\mathcal T) \stackrel{\curl}{\longrightarrow} \bm S_{k}^{(\bm{r'})}(\mathcal T) \stackrel{\div}{\longrightarrow} V_{k-1}^{(\bm{r''})}(\mathcal T)\stackrel{}{\longrightarrow} 0
    \end{equation}
    is a complex and exact, provided that the domain $\Omega$ is simply connected.
\end{proposition}
\begin{proof}
    Recall from the proof in \Cref{prop:derham} that at each vertex, the sum of the numbers of degrees of freedom of $V_{k+1}^{(\bm{r})}(\mathcal T) $ and $V_{k-1}^{(\bm{r''})}(\mathcal T)$ is $\binom{r_2+3}{2} + \binom{r_2+1}{2}$. On each edge, the sum of the numbers of degrees of freedom is
    \begin{equation*}
    a_{E,1} + a_{E,3} = \sum_{m = 0}^{r_1+1/2}(k+1+m-2r_2-3) + \sum_{m = 0}^{r_1-1/2}(k-1+m-2r_2+1)
    \end{equation*}
    Now consider the degrees of freedom of the space $\bm S_k^{(\bm r')}(\mathcal T)$ defined in \Cref{defn:generlized-stenberg}. At each vertex, the number of degrees of freedom is $\binom{r_2+3}{2}$. On each edge, the number of degrees of freedom is
    $$a_{E,2} = \sum_{m= 0}^{r_1+1/2}(k - 2r_2 -1 +m) + \sum_{m = 0}^{r_1 - 1/2}(k- 2r_2 -1 +m ).$$
    Since
    $$\binom{n}{2} - 2\binom{n-1}{2} + \binom{n-2}{2} = 1,$$
    it holds that
    $$a_{E,1} + a_{E,3} - a_{E,2} = \sum_{m = 0}^{r_1+1/2}(-1) + \sum_{m = 0}^{r_1-1/2}(1) = -1.$$
    Denote by $a_{K,1}, a_{K,2}, a_{K,3}$ the degrees of freedom defined inside element $K$. Since 
    $$a_{K,1} = \dim \Pcal_{k+1} - \binom{r_2+3}{2} - a_{E,1},$$
    $$a_{K,2} = \dim \Pcal_{k}^2 - \binom{r_2+2}{2} - a_{E,2},$$ and
    $$a_{K,3} = \dim \Pcal_{k+1} - \binom{r_2+1}{2} - a_{E,3},$$
    it holds that 
    $$a_{K,1} - a_{K,2} + a_{K,3} = 1 - 1 + 1 = 1.$$
    Consequently, by Euler's formula, it holds that
    \begin{equation*}
        \begin{split}
            \dim(V_{k+1}^{(\bm{r})}(\mathcal T)) - &
            \dim(\bm S_{k}^{(\bm{r}')}(\mathcal T)) + \dim(V_{k-1}^{(\bm{r}'')}(\mathcal T)) \\
            =
            & V -E + F = 1.
        \end{split}
    \end{equation*}

    It remains to show that the discrete kernel of $\div$ is just the discrete image of $\curl$. Suppose that for some $\bm v \in \bm S_{k}^{(\bm{r}')}(\mathcal T)$ such that $\div \bm v = 0$, then by the exactness of the continuous Stokes complex, there exists $\phi \in H^1(\Omega)$ such that $\curl \phi = \bm v$. Restricting this identity to each element $K$ immediately tells that $\phi$ is a polynomial of degree at most $k+1$ in each element $K$. It remains to show that $\phi$ satisfies the required continuity. Since $\phi \in H^1(\Omega)$ is piecewise smoothing, it follows that $\phi$ is continuous. 
    
    At each vertex, it follows from the $C^{r_2-1}$ continuity of $\bm v = \curl \phi$ that $\phi$ is of $C^{r_2}$ continuity. On each edge, the assumption yields that $$\frac{\partial}{\partial \bm n^{m}}(\curl \phi \cdot \bm t) = \frac{\partial}{\partial \bm n^{m+1}}(\phi)$$ is continuous (single-valued) across each internal edge $e$, for $m =0,1,\cdots,r_1 - 1/2$. Therefore, it holds that $\ker \div\subseteq \im \curl$ on the discrete level. The converse inclusion is from the definition of the complex. Hence it must hold that $\ker \div = \im \curl$ on the discrete level. Therefore, the complex \eqref{eq:smooth-de-rham2} is exact.
\end{proof}

\bibliographystyle{plain}
\bibliography{ref}
\end{document}